\documentclass[10pt]{article}
\usepackage{amsmath}
\usepackage{amsfonts}
\usepackage{amssymb}
\usepackage[T1]{fontenc}
\usepackage[utf8]{inputenc}
\usepackage{graphicx}
\usepackage{amsthm}
\usepackage{hyperref}
\usepackage{xcolor}
\usepackage{accents}
\usepackage{graphics,wrapfig, graphicx, mathrsfs,xcolor}
\usepackage{mathtools}
\usepackage{amssymb,amscd,graphicx,xcolor,subfig,tikz}
\usepackage{graphics}
\usepackage{indentfirst}
\usepackage{bm, enumerate,xcolor}
\usepackage[dvips]{epsfig}
\usepackage{latexsym}
\usepackage{float}
\usepackage{amsmath}
\usepackage[numbers,sort&compress,square]{natbib}

\usepackage[english]{babel}
\usepackage [autostyle, english = american]{csquotes}
\MakeOuterQuote{"}

\numberwithin{equation}{section}
\theoremstyle{plain}
\newtheorem*{theorem*}{Theorem}
\newtheorem*{lemma*}{Lemma}
\newtheorem*{corollary*}{Corollary}

\newtheorem{lemma}{Lemma}[section]
\newtheorem{corollary}[lemma]{Corollary}

\newtheorem{proposition}[lemma]{Proposition}

\theoremstyle{definition}

\newtheorem{remark}[lemma]{Remark}

\newtheorem*{example*}{Example}
\newtheorem*{er*}{Examples and Remarks}

\usepackage{geometry}
\newcommand{\dthth}{\partial_{\theta \theta}} 
 
\newcommand{\dth}{\partial_{\theta}}

\newcommand{\drr}{\partial_{rr}} 
\newcommand{\dr}{\partial_{r}} 
 
\newcommand{\dt}{\partial_{t}} 
\newcommand{\dy}{\partial_{y}}

\begin{document}

 \title{ Strong Ill-posedness of the 2d Incompressible Euler Equation in Critical Besov Spaces.}
\author{Karim R. Shikh Khalil}
 
\maketitle
\abstract{We prove strong ill-posedness  of the 2d incompressible Euler Equation for velocity field in the critical Besov Spaces $B^{1}_{\infty, q}$ for $1<q<\infty$. }

\section{Introduction}

The incompressible Euler equation:
\begin{equation}\label{Euler}
\begin{split}
\partial_t u &+  u\cdot\nabla u +\nabla p= 0, \\ 
&\nabla \cdot u=0, \\  
\end{split}
\end{equation}
is a fundamental model in fluid dynamics that describes the motion of ideal incompressible fluids with zero viscosity. Here, $u$ is the velocity and $p$ is the pressure that enforces the incompressibility constraint. The challenge of analyzing the Euler equation comes from  the fact that the equation is nonlinear and nonlocal (due to the incompressibility condition).

In this work, we study the well/ill-posedness of the initial value problem of the two-dimensional Euler equation. The literature on the initial value problem for the Euler equation is vast. The discussion here is not exhaustive, but we will focus on the works most relevant to this paper. The study of well-posedness started with the results of Lichtenstein~\cite{LL} and Gunther~\cite{G} proving well-posedness in H\"older spaces $C^{k,\alpha}$. Then, in two dimensions, Wolibner~\cite{Wo} and H\"older~\cite{HO} proved global well-posedness. Subsequently, Ebin and Marsden~\cite{EbM} proved well-posedness in Sobolev spaces $H^s$ for $s > \frac{d}{2} + 1$ in a compact domain, allowing for boundary. Bourguignon and Brezis~\cite{BB} extended this result to the Sobolev space $W^{s,p}$ for $s > \frac{d}{p} + 1$. Kato~\cite{Ka2} proved well-posedness in $H^s$ for  $s > \frac{d}{2} + 1$ on $\mathbb{R}^d$, and then Kato and Ponce~\cite{KaP} extended this to $W^{s,p}$ for $s > \frac{d}{p} + 1$. For Besov spaces, Vishik~\cite{MV1} proved global well-posedness in the critical Besov space $B^{2/p+1}_{p,1}$ with $1 < p < \infty$. Chae~\cite{C} proved well-posedness in $B^{d/p+1}_{p,1}$ with $1 < p < \infty$. Pak and Park~\cite{PP1,PP2} proved well-posedness in Besov spaces $B^1_{\infty,1}$ and  $B^{d+1}_{1,1}$. We remark that in all of the previously mentioned results on Besov spaces, the gradient of the velocity is bounded.

Regarding the strong ill-posedness results, we will first focus on  two dimensions. Bourgain and Li~\cite{BL1} proved strong ill-posedness  in the critical $H^2$ Sobolev space. In fact, they proved strong
 ill-posedness in dimensions  $d=2, 3$ for the Sobolev space $W^{\frac{d}{p}+1,p}$ for any $1 < p < \infty$, and for the Besov space $B^{\frac{d}{p}+1}_{p,q}$ for any $1 < p < \infty$, $1 < q \leq \infty$. Independently, Elgindi and Masmoudi~\cite{EM} and  Bourgain and Li~\cite{BL1} proved strong ill-posedness in $C^1$ and for all integer spaces. Later on, Elgindi and Jeong   \cite{EJ1} gave a simpler proof of the Sobolev ill-posedness in two dimensions. Jeong~\cite{J} gave an example of continuous loss of supercritical regularity, where the unique solution is initially in $W^{1,p}$ for $1 \leq p<2$, but continuously loses $W^{1,q(t)}$ regularity, where  $q(t)<p$ is decreasing. C\'{o}rdoba, Mart\'{i}nez-Zoroa, and O\.{z}a\'{n}ski~\cite{CMO} constructed unique solutions that exhibit gap loss of supercritical Sobolev regularity, where the solutions are initially in $H^{\beta}$ for $1 < \beta < 2$, but instantly leave $H^{\beta'}$ for all $\beta'>1+\frac{(3-\beta)(\beta-1)}{2-(\beta-1)^2}$.

Thus far, we have discussed strong ill-posedness results in two dimensions.  Bourgain and Li~\cite{BL1,BL} and Elgindi and Masmoudi~\cite{EM} results also hold in dimensions higher than two.  There are earlier results on strong ill-posedness. DiPerna and Lions~\cite{DL}  proved strong ill-posedness in supercritical spaces $W^{1,p}(\mathbb{T}^3)$ for $1 \leq p < \infty$, based on $2\frac{1}{2}$-dimensional shear flows. In addition, Bardos and Titi~\cite{BT} used the $2\frac{1}{2}$-dimensional shear flows to prove strong ill-posedness in the H\"older space $C^{\alpha}$. Further, Misio\symbol{170}ek and Yoneda~\cite{MY} also used   the $2\frac{1}{2}$-dimensional shear flows to prove strong ill-posedness in the log-Lip type space $LL^{\alpha}$ for   $0 < \alpha \leq 1$. Recently, Luo~\cite{L} proved strong ill-posedness for   velocity   in supercritical Sobolev spaces $H^s$ for any $0 < s < \frac{5}{2}$. More recently, Jeong, Mart\'{i}nez-Zoroa, and O\.{z}a\'{n}ski~\cite{JMO} proved instantaneous and continuous loss of supercritical Sobolev regularity for vorticity in $H^s$ for any $s \in (0, \frac{3}{2})$. There are also strong  ill-posedness results in Lorentz spaces by Kim and Jeong~\cite{KJ}, and recently by   Bang and Cheskidov~\cite{BC}. 

Up to this point, we have discussed strong ill-posedness results. There are also ill-posedness results in the sense of Hadamard, with regard to the continuity of the solution map. See the work of Himonas and   Misio\symbol{170}ek~\cite{HM} in Sobolev $H^s$ spaces, the work of  Cheskidov and Shvydkoy~\cite{ChS} in Besov spaces, and the work of  Misio\symbol{170}ek and Yoneda~\cite{MY2} in H\"older and Besov spaces.

\subsection{Main Result}
In this work, we study the 2d Euler equation in the critical  Besov spaces $B^{1}_{\infty, q}$ for $1<q<\infty$. To give more intuition about  this Besov regularity and put it in context, especially  with regards to the 2d Euler equation, we consider the following. From classical estimates for Besov spaces, see Vishik~\cite{MV1} and Bahouri, Chemin, and Danchin~\cite{BCD}, we have the following inclusion:    $$ B^{1}_{\infty, q} \subset LL^{1-\frac{1}{q}},$$ for $1\leq q\leq \infty$, where  $LL^{\alpha}$, with   $0 \leq \alpha \leq 1 $, is the Log-Lip-type space which is defined as follows: $$|u|_{LL^{\alpha}}=|u|_{L^{\infty}}+\sup_{0<|x-y|\leq \frac{1}{2}}\frac{|u(x)-u(y)|}{|x-y|(1+ |\log(|x-y|)|)^{\alpha}}$$

For the 2d Euler equation,  the Besov   $B^{1}_{\infty, q}$ for $1<q<\infty$ describes a remaining gap in regularity between what is known regarding well/ill-posedness (well-posedness: $C^{1+\alpha}$, $B^{1}_{\infty, 1}$ and ill-posedness: $C^{1}$, Lip) and the regularity threshold for uniqueness $B^{1}_{\infty, \infty}/$Log-Lip beyond which uniqueness may not hold. In this work, we prove strong ill-posedness of 2d Euler equation in this critical Besov $B^{1}_{\infty, q}$ for $1<q<\infty$. When the velocity is in these lower Besov regularity,  $B^{1}_{\infty, q}$ for $1<q<\infty$, the vorticity does not have to be bounded, though we still have uniqueness from the work of  Vishik \cite{MV2}. In order to obtain our ill-posedness result, we study the Besov $B^{1}_{\infty, q}$ space  in more detail. Namely, we introduce a simple yet useful approximation scheme by H\"older-type functions, where we couple the H\"older regularity exponent with an appropriate rescaling. This allows us to have quantitative bounds on the Besov norms in terms of  the H\"older regularity exponent. Using this approximation approach,  we identify two   classes of velocities that have the same $B^{1}_{\infty, q}$ norm, but with different gradients. Then, we leverage  their corresponding vorticity with the  Euler  dynamic in order to obtain the ill-posedness result through a new mechanism (see Subsection \ref{SketchProof} for more details). To control the Euler dynamic, we use the Biot-Savart law decomposition by Elgindi~\cite{E} to derive a leading order model for 2d Euler equation, where we first prove the ill-posedness.  Then, using a perturbative argument, we prove this result for 2d Euler equation. For a sketch of the main ideas in the proof, see Subsection \ref{SketchProof}. 

Finally, we remark that this work answers, in two dimensions, the remaining Besov end point $p=\infty$ left open in the work of Bourgain and Li~\cite{BL1}, which is also Problem 7 in "Geometric Hydrodynamics in Open Problems" by Khesin, Misio\symbol{170}ek, and Shnirelman~\cite{KMS}. This is achieved through a new ill-posedness mechanism, and improved understanding of the Besov $B^{1}_{\infty, q}$ space.

 \textbf{Statement of the Main Result:}
  \begin{theorem*} (Strong ill-posedness in ${B^{1}_{\infty,q}}$): For any $\delta>0$,  $1<q<\infty$, and $0<\alpha<1$,  there exist initial velocity data  $u_0^{\delta,\alpha,q} \in C^{1,\alpha}(\mathbb{R}^2)$ such that the solution to the 2d Euler equation  satisfies the following: 
  
  $$|u_0|_{B^{1}_{\infty,q}}=\delta  \hspace{0.3 cm } \text{but}  \hspace{0.3 cm }  \sup_{0\leq t \leq T(\alpha)}|u (t)|_{B^{1}_{\infty,q}} \geq   c \log(c |\log(\alpha)|) $$ 
  
  where 
  $T(\alpha)=c{\alpha^{1-\frac{1}{q}} \log(c |\log(\alpha)|) }$ and c is independent of $\alpha$.  
  \end{theorem*}
  \begin{remark}
Taking $\alpha \rightarrow 0$ implies that   $T(\alpha)=c{\alpha^{1-\frac{1}{q}} \log(c |\log(\alpha)|) }  \rightarrow 0$, and  thus we have  strong ill-posedness, since  $  |u|_{B^{1}_{\infty,q}} \rightarrow \infty$. 
   \end{remark}
    \begin{remark}
 We observe that when $q=1$, the   ${B^{1}_{\infty,1}}$ case, the time scale is   $T(\alpha)=c \log(c |\log(\alpha)|)  \nrightarrow 0 $, as $\alpha \rightarrow 0$. This is consistent with the well-posedness and  the absence of norm inflation in ${B^{1}_{\infty,1}}$, see Vishik \cite{MV1} and  Pak and Park \cite{PP1}. 
   \end{remark}
    \begin{corollary*} 
    For any $\delta>0$,  $1<q<\infty$, and $0<\alpha<1$,  there exist smooth initial velocity   $u_0^{\delta,\alpha,q} \in C^{\infty}(\mathbb{R}^2)$ such that the solution to the 2d Euler equation  satisfies the following: 
  
  $$|u_0|_{B^{1}_{\infty,q}}=\delta  \hspace{0.3 cm } \text{but}  \hspace{0.3 cm }  \sup_{0\leq t \leq T(\alpha)}|u(t)|_{B^{1}_{\infty,q}} \geq   c \log(c |\log(\alpha)|) $$ 
  
  where 
  $T(\alpha)=c{\alpha^{1-\frac{1}{q}} \log(c |\log(\alpha)|) }$ and c is independent of $\alpha$.
      \end{corollary*}
 
\subsection{Notation}

In this paper, we will be mostly working in  polar coordinates on $\mathbb{R}^2$. Given a Cartesian point $x=(x_1,x_2)$ recall we have: 
$$
r=\sqrt{x_1^2+x_2^2},\, \text{and} \,\, \theta=\arctan{(\frac{x_2}{x_1})}.
$$

We will use $|f|_{L^{\infty}}$ and $|f|_{C^{\alpha}}$ to denote the usual $L^{\infty}$ and $C^{\alpha}$ H\"older norm for $0<\alpha<1$.  In addition, we will also use some  H\"older-type norms introduced in the work of Elgindi and Jeong \cite{EJ3}: 
$$
|f|_{\mathring{C}^{1,\alpha}}=   |f|_{C^{\alpha}}+   |\dth  f|_{C^{\alpha}}+|r \dr f|_{C^{\alpha}}. 
$$

We will use $\Delta_{h}$ to denote the finite difference operator: $\Delta_{h}f=f(x+h)-f(x).$
Note that $\Delta$ will still denote the standard Laplacian: $\Delta f=\partial_{11}f+\partial_{22}f=\drr f+\frac{1}{r} \dr f+\frac{1}{r^2}\dthth f.$

 \subsection{Sketch of the Proof}\label{SketchProof}

 There are essentially three main ideas that will go into the proof: 
 
  \textbullet \textit{\, Computing Besov $B^{1}_{\infty,q}$ norm using  H\"older  type functions and identifying two classes of functions with different gradient but with the same Besov $B^{1}_{\infty,q}$ norm:} First, we approximate functions in the Besov space by H\"older  type ones, where we couple the H\"older regularity exponent  with appropriate rescaling.   This allows us to have a quantitive estimate on the Besov norm in terms of the H\"older regularity exponent. The first type of functions is of the following form: \begin{equation}\label{BesovType1}
f(r)=
 \frac{1}{\alpha^{\frac{1}{k}}}r (1- r^{\alpha}),
\end{equation}
  for  $1\leq k \leq \infty$. Then, the Besov norm, for $1\leq q< \infty$, is the following:
 $$|f|_{B^{1}_{\infty,q}} = c_0 \alpha^{1-\frac{1}{q}-\frac{1}{k}}.$$ 
For example when $k=2$, we have  
\begin{equation}\label{BesovEx1}
f(r)=
 \frac{1}{\alpha^{\frac{1}{2}}}r (1- r^{\alpha}). 
 \end{equation}
Therefore, for $1\leq q<\infty$, this function will have size of $$|f|_{B^{1}_{\infty,q}}= c\alpha^{\frac{1}{2}-\frac{1}{q}}.$$
 We observe the following: 
  \begin{align*}
 |f|_{B^{1}_{\infty,2}}&=c, \,\,  \text{\textit{independent} of $\alpha$, when} \,\, q=2.  \\  
|f|_{B^{1}_{\infty,q}}&=c\alpha^{\frac{1}{2}-\frac{1}{q}} \rightarrow \infty, \,\, \text{as} \,\, \alpha \rightarrow 0,  \,\, \text{when} \,\, 1 \leq q<2.\\
|f|_{B^{1}_{\infty,q}}&=c\alpha^{\frac{1}{2}-\frac{1}{q}} \rightarrow 0, \,\,  \,\, \,  \text{as} \,\, \alpha \rightarrow 0,  \,\, \text{when} \,\, 2 <q<\infty. 
 \end{align*} 
Hence, this function is $O(1)$ in  $B^{1}_{\infty,2}$. Whereas when $1 \leq q<2$, it is large in  $B^{1}_{\infty,q}$, and  when $2 <q<\infty$, it is small in  $B^{1}_{\infty,q}$.  See Section \ref{Besovf} for more details.

Next, we identify the second type of functions with the same Besov norm as their corresponding functions in \eqref{BesovType1}, but with a different gradient which is going to be bounded and  oscillating. Namely, we observe, for $1\leq k \leq \infty$, that the functions:  
\begin{equation}\label{BesovType2}
g(r)=r^{1+\alpha} \sin(\alpha^{1-\frac{1}{k}}\log(r))),
\end{equation}
 have the following  Besov norm:    
 $$ |g|_{B^{1}_{\infty,q}} =   c_0 \alpha^{1-\frac{1}{q}-\frac{1}{k}} $$
  To make it more concrete, as in the previous examples, consider the following function:
    \begin{equation}\label{BesovEx2}
    g(r)=r^{1+\alpha} \sin(\alpha^{\frac{1}{2}}\log(r))).
    \end{equation}
  We obseve the following: 
  \begin{align*}
 |g|_{B^{1}_{\infty,2}}&=c, \,\,  \text{\textit{independent} of $\alpha$, when} \,\, q=2.  \\  
|g|_{B^{1}_{\infty,q}}&=c\alpha^{\frac{1}{2}-\frac{1}{q}} \rightarrow \infty, \,\, \text{as} \,\, \alpha \rightarrow 0,  \,\, \text{when} \,\, 1 \leq q<2.\\
|g|_{B^{1}_{\infty,q}}&=c\alpha^{\frac{1}{2}-\frac{1}{q}} \rightarrow 0, \,\,  \,\, \,  \text{as} \,\, \alpha \rightarrow 0,  \,\, \text{when} \,\, 2 <q<\infty. 
 \end{align*} 
Hence, we have similar sizes as in \eqref{BesovEx1}  (See Section \ref{Besovf} for more details). From the two type of functions we considered so far, we observe that if we take   
 $$
 z(r)=\frac{1}{\alpha^{\frac{1}{2}}}r (1- r^{\alpha})  \sin(\alpha^{\frac{1}{2}}\log(r))).  \, \text{Then, we have } |z|_{B^{1}_{\infty,2}}=\frac{c}{\alpha^{\frac{1}{2}}}.
 $$
Heuristically, with the appropriate  rescaling, one can think of  functions in ${B^{1}_{\infty,q}}$ as either having a large gradient or an oscillating gradient, but not both. This will be a key into the ill-posedness mechanism. 
  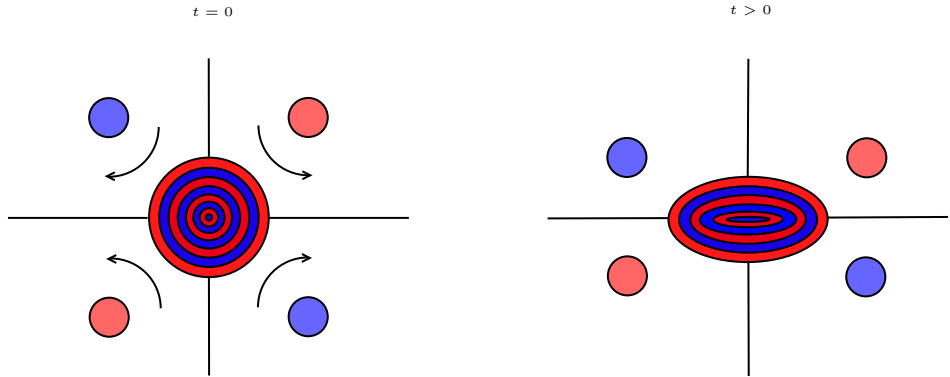
\begin{figure}[ht]
\centering

\tikzset{every picture/.style={line width=0.75pt}} 

\begin{tikzpicture}[x=0.75pt,y=0.75pt,yscale=-1,xscale=1]

\draw  [fill={rgb, 255:red, 255; green, 0; blue, 0 }  ,fill opacity=0.6 ] (240.13,89.99) .. controls (240.13,84.58) and (244.52,80.19) .. (249.94,80.19) .. controls (255.35,80.19) and (259.73,84.58) .. (259.73,90) .. controls (259.73,95.41) and (255.34,99.79) .. (249.93,99.79) .. controls (244.52,99.79) and (240.13,95.4) .. (240.13,89.99) -- cycle ;
\draw  [draw opacity=0][line width=3] [line join = round][line cap = round] (160.29,26.64) .. controls (159.06,29.93) and (156.77,32.73) .. (154.79,35.64) ;
\draw    (230.39,140.4) -- (300.39,140.4) ;
\draw    (200.25,219.5) -- (200.16,170.13) ;
\draw    (99.43,140.4) -- (169.43,140.4) ;
\draw    (200.16,109.67) -- (200.07,60.3) ;
\draw  [fill={rgb, 255:red, 255; green, 0; blue, 0 }  ,fill opacity=0.6 ] (140.38,190.24) .. controls (140.38,184.83) and (144.77,180.44) .. (150.19,180.44) .. controls (155.6,180.44) and (159.98,184.83) .. (159.98,190.25) .. controls (159.98,195.66) and (155.59,200.04) .. (150.18,200.04) .. controls (144.77,200.04) and (140.38,195.65) .. (140.38,190.24) -- cycle ;
\draw  [fill={rgb, 255:red, 0; green, 0; blue, 255 }  ,fill opacity=0.6 ] (240.13,189.99) .. controls (240.13,184.58) and (244.52,180.19) .. (249.94,180.19) .. controls (255.35,180.19) and (259.73,184.58) .. (259.73,190) .. controls (259.73,195.41) and (255.34,199.79) .. (249.93,199.79) .. controls (244.52,199.79) and (240.13,195.4) .. (240.13,189.99) -- cycle ;
\draw  [fill={rgb, 255:red, 0; green, 0; blue, 255 }  ,fill opacity=0.6 ] (140.13,90.07) .. controls (140.13,84.65) and (144.52,80.27) .. (149.94,80.27) .. controls (155.35,80.27) and (159.73,84.66) .. (159.73,90.07) .. controls (159.73,95.49) and (155.34,99.87) .. (149.93,99.87) .. controls (144.52,99.87) and (140.13,95.48) .. (140.13,90.07) -- cycle ;
\draw  [fill={rgb, 255:red, 255; green, 0; blue, 0 }  ,fill opacity=0.9 ] (170.16,140.13) .. controls (170.16,123.56) and (183.59,110.13) .. (200.16,110.13) .. controls (216.73,110.13) and (230.16,123.56) .. (230.16,140.13) .. controls (230.16,156.7) and (216.73,170.13) .. (200.16,170.13) .. controls (183.59,170.13) and (170.16,156.7) .. (170.16,140.13) -- cycle ;
\draw  [fill={rgb, 255:red, 0; green, 0; blue, 255 }  ,fill opacity=0.9 ] (175.25,140.13) .. controls (175.25,126.37) and (186.4,115.21) .. (200.16,115.21) .. controls (213.92,115.21) and (225.08,126.37) .. (225.08,140.13) .. controls (225.08,153.89) and (213.92,165.04) .. (200.16,165.04) .. controls (186.4,165.04) and (175.25,153.89) .. (175.25,140.13) -- cycle ;
\draw  [fill={rgb, 255:red, 255; green, 0; blue, 0 }  ,fill opacity=0.9 ] (179.96,140.13) .. controls (179.96,128.97) and (189.01,119.93) .. (200.16,119.93) .. controls (211.32,119.93) and (220.36,128.97) .. (220.36,140.13) .. controls (220.36,151.28) and (211.32,160.33) .. (200.16,160.33) .. controls (189.01,160.33) and (179.96,151.28) .. (179.96,140.13) -- cycle ;
\draw  [fill={rgb, 255:red, 0; green, 0; blue, 255 }  ,fill opacity=0.9 ] (184.5,140.13) .. controls (184.5,131.48) and (191.51,124.46) .. (200.16,124.46) .. controls (208.81,124.46) and (215.83,131.48) .. (215.83,140.13) .. controls (215.83,148.78) and (208.81,155.79) .. (200.16,155.79) .. controls (191.51,155.79) and (184.5,148.78) .. (184.5,140.13) -- cycle ;
\draw  [fill={rgb, 255:red, 255; green, 0; blue, 0 }  ,fill opacity=0.9 ] (188.64,140.13) .. controls (188.64,133.76) and (193.8,128.6) .. (200.16,128.6) .. controls (206.53,128.6) and (211.68,133.76) .. (211.68,140.13) .. controls (211.68,146.49) and (206.53,151.65) .. (200.16,151.65) .. controls (193.8,151.65) and (188.64,146.49) .. (188.64,140.13) -- cycle ;
\draw  [fill={rgb, 255:red, 0; green, 0; blue, 255 }  ,fill opacity=0.9 ] (192.31,140.13) .. controls (192.31,135.79) and (195.82,132.27) .. (200.16,132.27) .. controls (204.5,132.27) and (208.02,135.79) .. (208.02,140.13) .. controls (208.02,144.46) and (204.5,147.98) .. (200.16,147.98) .. controls (195.82,147.98) and (192.31,144.46) .. (192.31,140.13) -- cycle ;
\draw  [fill={rgb, 255:red, 255; green, 0; blue, 0 }  ,fill opacity=0.9 ] (195.55,140.13) .. controls (195.55,137.58) and (197.61,135.51) .. (200.16,135.51) .. controls (202.71,135.51) and (204.78,137.58) .. (204.78,140.13) .. controls (204.78,142.68) and (202.71,144.74) .. (200.16,144.74) .. controls (197.61,144.74) and (195.55,142.68) .. (195.55,140.13) -- cycle ;
\draw  [fill={rgb, 255:red, 0; green, 0; blue, 255 }  ,fill opacity=0.9 ] (198.57,140.13) .. controls (198.57,139.25) and (199.28,138.53) .. (200.16,138.53) .. controls (201.04,138.53) and (201.76,139.25) .. (201.76,140.13) .. controls (201.76,141.01) and (201.04,141.72) .. (200.16,141.72) .. controls (199.28,141.72) and (198.57,141.01) .. (198.57,140.13) -- cycle ;
\draw  [fill={rgb, 255:red, 255; green, 0; blue, 0 }  ,fill opacity=0.6 ] (520.13,110.24) .. controls (520.13,104.83) and (524.52,100.44) .. (529.94,100.44) .. controls (535.35,100.44) and (539.73,104.83) .. (539.73,110.25) .. controls (539.73,115.66) and (535.34,120.04) .. (529.93,120.04) .. controls (524.52,120.04) and (520.13,115.65) .. (520.13,110.24) -- cycle ;
\draw    (510.11,140.23) -- (570.66,139.92) ;
\draw    (470.75,219.75) -- (470.58,162.56) ;
\draw    (369.93,140.65) -- (430.5,140.5) ;
\draw    (470.24,119.77) -- (470.57,60.55) ;
\draw  [fill={rgb, 255:red, 255; green, 0; blue, 0 }  ,fill opacity=0.6 ] (400.13,169.49) .. controls (400.13,164.08) and (404.52,159.69) .. (409.94,159.69) .. controls (415.35,159.69) and (419.73,164.08) .. (419.73,169.5) .. controls (419.73,174.91) and (415.34,179.29) .. (409.93,179.29) .. controls (404.52,179.29) and (400.13,174.9) .. (400.13,169.49) -- cycle ;
\draw  [fill={rgb, 255:red, 0; green, 0; blue, 255 }  ,fill opacity=0.6 ] (519.63,169.99) .. controls (519.63,164.58) and (524.02,160.19) .. (529.44,160.19) .. controls (534.85,160.19) and (539.23,164.58) .. (539.23,170) .. controls (539.23,175.41) and (534.84,179.79) .. (529.43,179.79) .. controls (524.02,179.79) and (519.63,175.4) .. (519.63,169.99) -- cycle ;
\draw  [fill={rgb, 255:red, 0; green, 0; blue, 255 }  ,fill opacity=0.6 ] (399.88,110.07) .. controls (399.88,104.65) and (404.27,100.27) .. (409.69,100.27) .. controls (415.1,100.27) and (419.48,104.66) .. (419.48,110.07) .. controls (419.48,115.49) and (415.09,119.87) .. (409.68,119.87) .. controls (404.27,119.87) and (399.88,115.48) .. (399.88,110.07) -- cycle ;
\draw  [fill={rgb, 255:red, 255; green, 0; blue, 0 }  ,fill opacity=0.9 ] (430.51,141.48) .. controls (430.41,129.67) and (448.2,119.95) .. (470.24,119.77) .. controls (492.28,119.6) and (510.22,129.04) .. (510.32,140.85) .. controls (510.41,152.67) and (492.62,162.39) .. (470.58,162.56) .. controls (448.54,162.74) and (430.6,153.3) .. (430.51,141.48) -- cycle ;
\draw  [fill={rgb, 255:red, 0; green, 0; blue, 255 }  ,fill opacity=0.9 ] (435.8,141.17) .. controls (435.8,131.97) and (451.3,124.51) .. (470.41,124.51) .. controls (489.53,124.51) and (505.02,131.97) .. (505.02,141.17) .. controls (505.02,150.37) and (489.53,157.82) .. (470.41,157.82) .. controls (451.3,157.82) and (435.8,150.37) .. (435.8,141.17) -- cycle ;
\draw  [fill={rgb, 255:red, 255; green, 0; blue, 0 }  ,fill opacity=0.9 ] (441.48,141.4) .. controls (441.43,134.65) and (454.34,129.09) .. (470.32,128.96) .. controls (486.3,128.84) and (499.29,134.2) .. (499.35,140.94) .. controls (499.4,147.68) and (486.49,153.25) .. (470.51,153.37) .. controls (454.53,153.5) and (441.53,148.14) .. (441.48,141.4) -- cycle ;
\draw  [fill={rgb, 255:red, 0; green, 0; blue, 255 }  ,fill opacity=1 ] (446.2,141.36) .. controls (446.17,137.18) and (456.98,133.71) .. (470.35,133.6) .. controls (483.72,133.5) and (494.59,136.8) .. (494.62,140.98) .. controls (494.65,145.16) and (483.84,148.63) .. (470.47,148.73) .. controls (457.1,148.84) and (446.24,145.54) .. (446.2,141.36) -- cycle ;
\draw  [fill={rgb, 255:red, 255; green, 0; blue, 0 }  ,fill opacity=0.9 ] (453.02,141.3) .. controls (453,139.15) and (460.77,137.33) .. (470.38,137.26) .. controls (479.99,137.18) and (487.79,138.87) .. (487.81,141.03) .. controls (487.82,143.19) and (480.05,145) .. (470.44,145.08) .. controls (460.84,145.15) and (453.03,143.46) .. (453.02,141.3) -- cycle ;
\draw  [fill={rgb, 255:red, 0; green, 0; blue, 255 }  ,fill opacity=0.9 ] (459.76,141.25) .. controls (459.75,140.54) and (464.52,139.93) .. (470.4,139.89) .. controls (476.29,139.84) and (481.06,140.38) .. (481.06,141.08) .. controls (481.07,141.79) and (476.31,142.4) .. (470.42,142.45) .. controls (464.54,142.49) and (459.76,141.96) .. (459.76,141.25) -- cycle ;
\draw  [draw opacity=0] (224.75,185.25) .. controls (224.75,185.25) and (224.75,185.25) .. (224.75,185.25) .. controls (224.75,171.44) and (235.94,160.25) .. (249.75,160.25) .. controls (249.75,160.25) and (249.75,160.25) .. (249.75,160.25) -- (249.75,185.25) -- cycle ; \draw   (224.75,185.25) .. controls (224.75,185.25) and (224.75,185.25) .. (224.75,185.25) .. controls (224.75,171.44) and (235.94,160.25) .. (249.75,160.25) .. controls (249.75,160.25) and (249.75,160.25) .. (249.75,160.25) ;  
\draw  [draw opacity=0] (175,94.75) .. controls (175,94.75) and (175,94.75) .. (175,94.75) .. controls (175,108.56) and (163.81,119.75) .. (150,119.75) -- (150,94.75) -- cycle ; \draw   (175,94.75) .. controls (175,94.75) and (175,94.75) .. (175,94.75) .. controls (175,108.56) and (163.81,119.75) .. (150,119.75) ;  
\draw  [draw opacity=0] (151,160.75) .. controls (164.81,160.75) and (176,171.94) .. (176,185.75) -- (151,185.75) -- cycle ; \draw   (151,160.75) .. controls (164.81,160.75) and (176,171.94) .. (176,185.75) ;  
\draw   (247,117) -- (251,119) -- (247,121) ;
\draw   (153.75,159) -- (149.75,161) -- (153.75,163) ;
\draw   (246.75,158.5) -- (250.75,160.5) -- (246.75,162.5) ;
\draw  [draw opacity=0] (250,119) .. controls (250,119) and (250,119) .. (250,119) .. controls (236.19,119) and (225,107.81) .. (225,94) -- (250,94) -- cycle ; \draw   (250,119) .. controls (250,119) and (250,119) .. (250,119) .. controls (236.19,119) and (225,107.81) .. (225,94) ;  
\draw   (153,117.5) -- (149,119.5) -- (153,121.5) ;

\draw (190.25,33.25) node [anchor=north west][inner sep=0.75pt]  [font=\tiny] [align=left] {$\displaystyle t=0$};
\draw (459.75,32.25) node [anchor=north west][inner sep=0.75pt]  [font=\tiny] [align=left] {$\displaystyle t >0$};

\end{tikzpicture}
\caption{Sketch of initial data and the ill-posedness mechanism }\label{1-1}
\end{figure}
 
\textbullet \textit{\, Ill-posedness mechanism:}  To explain the ill-posedness mechanism and make it more concrete, we will focus on the ${B^{1}_{\infty,2}}$ case. The functions we have discussed so far are at the level of the velocity field. Now, we will discuss their corresponding vorticity. First, we observe that if we consider a compactly supported vorticity of the following form: $$\omega_f(r,\theta)=\alpha^{\frac{1}{2}}r^{\alpha}\sin(2\theta),$$
It will generate a hyperbolic flow around the origin, and the velocity will be of approximately the following form: $$u(\omega_f)\approx \frac{1 }{\alpha^{\frac{1}{2}}} r  (1- r^{\alpha}), $$ which will have large gradient, but, as discussed, it will  still be $O(1)$ in   ${B^{1}_{\infty,2}}$. Now we consider another compactly supported  vorticity of the following form: 
$$\omega_g(r,\theta)=r^{\alpha} \sin(\alpha^{\frac{1}{2}}\log(r)).$$
This vorticity will generate a velocity with a bounded oscillating gradient of the form: $$u(\omega_g) \approx r^{1+\alpha} \sin(\alpha^{\frac{1}{2}}\log(r))).  $$  
Similarly, as discussed, it will still be $O(1)$ in   ${B^{1}_{\infty,2}}$. Our initial data   will be
$$
\omega_0=\omega_{g}+\omega_{f}=r^{\alpha} \sin(\alpha^{\frac{1}{2}}\log(r))+ \alpha^{\frac{1}{2}}r^{\alpha}\sin(2\theta).
$$
The ill-posedness will arise because $\omega_{f}= \alpha^{\frac{1}{2}}r^{\alpha}\sin(2\theta)$ will generate a hyperbolic flow around the origin that will stretch the radial vorticity $\omega_{g}$ along the $x$-axis and shrink it along the $y$-axis. Thus, making the level sets of the vorticity $\omega_{g}$, which were initially   circular,  look approximately more like an ellipse. Hence,  the vorticity coming from $\omega_{g}$ will develop an angular dependence.  See Figure 1.  Hence, at $O(t)$ time, it will generate a term of the following type: 
$$
\omega \approx r^{\alpha} \sin(\alpha^{\frac{1}{2}}\log(r))\cos(2\theta). 
$$ Now, this vorticity will generate a velocity of approximately the following form: $$u(\omega)\approx \frac{1}{\alpha^{\frac{1}{2}}} r  (1- r^{\alpha}) \sin(\alpha^{\frac{1}{2}}\log(r)),$$ which will be of size: $$|u(\omega)|_{B^{1}_{\infty,2}}= O(\frac{1}{\alpha^{\frac{1}{2}}}).$$ See Section \ref{LOME} for more details. We remark that the dynamic generated from the full initial data $\omega_0$ will essentially be hyperbolic due to the  $\omega_{f}$ term. This is because, although $\omega_g$ and $\omega_f$ both have the same initial size in $B^{1}_{\infty,2}$, the gradient of the velocity field from $\omega_{f}$ will be $O(\alpha^{-\frac{1}{2}})$, whereas the gradient of the velocity field from $\omega_{g}$ will be $O(1)$. Thus, on our time scale, it will not have much effect on the dynamics.

\textbullet \textit{\, Capturing the main dynamic through a leading order model:} To control the dynamic and to prove ill-posedness, we first use the Biot-Savart law decomposition by Elgindi \cite{E}, see Proposition \ref{ElgindiCcirc}, to derive   a leading order model for the 2d Euler  (see Section \ref{LOM}). This idea was used in the  previous work with Elgindi  \cite{ESK} to prove ill-posedness in $L^{\infty}$ for the Euler equation with Riesz forcing. The idea here is to consider a model where we keep only the main terms of the velocity field identified   by Elgindi. This will reduce the complexity of Euler system and make it easier to capture the main dynamic. We prove the  ill-posedness of the leading order model in Section \ref{LOME}. Then, we use a perturbative argument to show that this ill-posedness of the leading order model will still carry to the full 2d Euler (this will be done in Section \ref{RemainderrEstS} and  Section \ref{Main}).

 \subsection{Organization}
Section \ref{Prelim} is preliminary results where we recall Biot-Savart law decomposition.  In Section \ref{LOM}, we introduce  the leading order model. In Section  \ref{Besovf}, we discuss  the Besov $B^{1}_{\infty,q}$ norm approximation with H\"older type functions and two classes of functions discussed   in Subsection \ref{SketchProof}. In Section \ref{LOME}, we obtain our main estimates for the ill-posedness of the leading order model.  In Section \ref{UsefulLOME}, we  obtain useful estimates for the leading order model  which we will use in Section \ref{RemainderrEstS} to control the error between the leading order model  and the full 2d Euler equation. Finally, in Section \ref{Main}, we prove our main result.

 \section{Preliminary Results:}\label{Prelim}
Here we recall the Biot Savart law decomposition by Elgindi \cite{E}, which we will use to derive a leading order model for 2d Euler equation, see Section \ref{LOM}.

\begin{proposition}\label{ElgindiCcirc}(Elgindi \cite{E}: Biot Savart law decomposition) Let vorticity $\omega \in C^{\alpha}(\mathbb{R}^{2})\cap {\mathring{C}^{1,\alpha}}(\mathbb{R}^{2}) $ with compact support. We set  the stream function as follows: $\psi(r,\theta)=r^2\Psi(r,\theta)$. Then we can write 
 
 $$
 \Psi(\omega)(r,\theta)=L^{s}_{12}(\omega)(r) \sin(2\theta)+L^{c}_{12}(\omega)(r) \cos(2\theta)+\Psi_r(\omega)(r,\theta)
 $$
 where
 
 $$L^s_{12}(\omega)(r)= \frac{1}{4\pi} \int_{r}^{\infty} \int_0^{2\pi} \frac{\omega(\theta,s)\sin(2\theta)}{s} \, d\theta \, ds, \, \, \text{and} \, \, L^c_{12}(\omega)(r)=\frac{1}{4\pi} \int_{r}^{\infty} \int_0^{2\pi} \frac{\omega(\theta,s)\cos(2\theta)}{s} \, d\theta \, ds, $$
  and 
  
  $$|\Psi_r(\omega)|_{C^{\alpha}}, |r\dr \Psi_r(\omega)|_{C^{\alpha}} \leq c |\omega|_{C^{\alpha}\cap {\mathring{C}^{1,\alpha}}} $$
where $c$ is independent of $\alpha$. 
\end{proposition}
\begin{remark}
The Biot-Savart law decomposition in \cite{E} is for the 3d Euler equation.  In the case of 2d Euler, the Biot-Savart law decomposition simplifies. There are several variants of it. For the $L^{\infty}$ version, see  \cite{Eremarks} and Lemma 5.1 in \cite{EJ2}. For $L^{2}$ energy version, see Theorem 4.23 and Theorem 4.24 in \cite{DE}. The version  we are using  essentially follows from the mentioned works and Theorem 2 in \cite{EJ3}.  Finally, we remark that this can be considered as a generalization of the Key-Lemma of Kiselev and \v{S}ver\'{a}k \cite{KS}. 
\end{remark}
\begin{remark}\label{PsiReminderEst}
When we do the remainder estimate in Section \ref{RemainderrEstS}, we will write $$\psi_r:=\psi(r,\theta)-r^2L^{s}_{12}(\omega)(r) \sin(2\theta)-r^2L^{c}_{12}(\omega)(r) \cos(2\theta) =r^2\Psi(r,\theta)-r^2L^{s}_{12}(\omega)(r) \sin(2\theta)-r^2L^{c}_{12}(\omega)(r) \cos(2\theta),$$
to denote the part the stream function that will give rise to  the gradient of velocity field which has H\"older estimate with a constant of $c$ independent of $\alpha$. 
\end{remark}

\section{Leading Order Model}\label{LOM}
In this section we present the leading order model for the 2d Euler equation. Recall that the 2d Euler equation in vorticity form can be written as:  
\begin{equation}
\partial_t \omega + u \cdot \nabla \omega = 0,
\end{equation}
where the velocity is given by $u = \nabla^{\perp} \psi$, with $\psi$ denoting the stream function, which solves $\Delta \psi = \omega$. In polar coordinates, the 2d Euler equation takes the following form:  
\begin{equation}
\partial_t \omega + \frac{1}{r} \, \partial_r \psi \, \partial_\theta \omega 
- \frac{1}{r} \, \partial_\theta \psi \, \partial_r \omega = 0.
\end{equation}
From Proposition \ref{ElgindiCcirc}, we can  rewrite the stream function as follows:

$$\psi(r,\theta)=r^2\Psi(r,\theta),$$ where  
 
 $$
 \Psi(\omega)(r,\theta)=L^{s}_{12}(\omega)(r) \sin(2\theta)+L^{c}_{12}(\omega)(r) \cos(2\theta)+\Psi_r(\omega)(r,\theta),
 $$
 and the operator $L^s_{12}$ and $L^c_{12}$ are defined as follows: 
 \begin{align}
 L^{s}_{12} (\omega)(r)=&\frac{1}{4\pi} \int_r^{\infty} \int_0^{2\pi}  \frac{ 1}{s} \omega(s,\theta) \sin(2\theta) d \theta \, ds, \label{L12s}\\  
 L^{c}_{12} (\omega)(r)=&\frac{1}{4\pi} \int_r^{\infty} \int_0^{2\pi}  \frac{ 1}{s} \omega(s,\theta) \cos(2\theta) d \theta \, ds. \label{L12c} 
 \end{align} 
The leading order model we consider is the following:

\begin{equation} \label{LOMeqExp}
 \dt \Omega+  2 L^s_{12}(\Omega) \sin(2\theta)  \dth \Omega- 2 rL^s_{12}(\Omega) \cos(2\theta) \dr \Omega =0. 
\end{equation} 
In our previous work with Elgindi \cite{ESK}, we derived a similar leading order model   for the Euler equation   with Riesz forcing to prove ill-posedness in $L^{\infty}$ .

\section{Besov  $B^{1}_{\infty,q}$ Norm using H\"older-type Functions }\label{Besovf}
In this section, we approximate functions in the Besov space  $B^{1}_{\infty,q}$ by H\"older-type  functions, where we couple the H\"older regularity exponent with appropriate rescaling. We do this in order to have a quantitative estimate on the Besov norm for these functions. Besov spaces can be defined in several equivalent ways, commonly using  Littlewood-Paley theory, but due to the nature of our approach where we work in the physical space, we follow the definition of the Besov norm using finite difference. This definition can be found in the book on Sobolev spaces by  Leoni \cite{GL}.

Following  \cite{GL},  The general Besov $B^{s}_{p,q}$ space using finite difference is defined as follows:  Given a function $u:\mathbb{R}^{d} \rightarrow \mathbb{R}$,  and for every $h \in \mathbb{R}^{d}$, $x \in \mathbb{R}^d$, and  $m \in \mathbb{N}$. The finite difference operator $\Delta_{h}$ is defined as follows:
$$
\Delta_h u(x)= u(x+h)-u(x).$$
Note, we use subscript $\Delta_h $ to distinguish it  from the Laplacian operator $\Delta$.  Then $\Delta_h^m$ is define as follows:
$$
\Delta^m_h u(x)= \Delta_h ( \Delta^{m-1}_h(u(x))).$$
 Using the notation that $\lfloor s \rfloor $ to denote  the integer part of $s$, we have for $s>0$,  
$$
|u|_{B^{s}_{p,q}}= \bigg(\int_{\mathbb{R}^n} |\Delta_h^{\lfloor s \rfloor +1} u|^q_{L^p}\frac{1}{|h|^{d+sq}}   dh \bigg)^{\frac{1}{q} \,}.
$$ 
 Note we wrote the integrality index $p$ and the summability index $q$ as subscript. In this paper, we  will be mostly working with $s=1$, $p=\infty$, and $1\leq q<\infty$. Hence, in this case we have:  
$$
|u|_{B^{1}_{\infty,q}}= \bigg(\int_{\mathbb{R}^n} |\Delta_h^{2} u|^q_{L^{\infty}}\frac{1}{|h|^{d+sq}}   dh \bigg)^{\frac{1}{q} \,},
$$ 
where
$$
\Delta^2_h u(x)= u(x+2h)-2 u(x+h)+u(x).
$$

Now we are ready to state and prove the lemmas for this section.  
\begin{lemma}\label{Type1}  Let  $0<\alpha<1$ and consider for any $1\leq k \leq \infty$ the function $f:[0,1] \rightarrow \mathbb{R} $ defined as follows: 

$$f(r)=   
 \frac{c}{\alpha^{\frac{1}{k}}}r (1- r^{\alpha}).
 $$
Then, Besov norm ${B^{1}_{\infty,q}}$  of  $f$, for $1 \leq q < \infty  $, satisfies the following:
$$
  c_1 \alpha^{1-\frac{1}{q}-\frac{1}{k}}  \leq |f|_{B^{1}_{\infty,q}} \leq   c_0 \alpha^{1-\frac{1}{q}-\frac{1}{k}} ,
$$
for $c_0$ and $c_1$ independent of $\alpha$.

\end{lemma}

\begin{proof}

For the upper bound, we observe that 
$$
|\Delta_h^{2} f|_{L^{\infty}}\leq c_0  |h|^2|f''(h)| \leq c_0 \frac{\alpha}{\alpha^{\frac{1}{k}}}|h|^{1+\alpha}.
$$
Hence,
$$
|f|^q_{B^{1}_{\infty,q}} \leq   \int_{0}^1 |\Delta_h^{2} f|^q_{L^{\infty}}\frac{1}{|h|^{1+q}}   dh \leq  c  \int_{0}^1  \frac{\alpha^q}{\alpha^{\frac{q}{k}}}|h|^{q+q\alpha}  \frac{1}{|h|^{1+q}}   dh  =  c  \int_{0}^1  \frac{\alpha^q}{\alpha^{\frac{q}{k}}}|h|^{q\alpha-1}    dh  = c_0   \alpha^{q-1-\frac{q}{k}}.$$
Thus,
$$
|f|_{B^{1}_{\infty,q}} \leq  c_0 \alpha^{1-\frac{1}{q}-\frac{1}{k}}.  
$$
Therefore, we have the upper bound. Now we would like to obtain the lower bound. Here we have

$$
|\Delta^2_h f|_{L^{\infty}}\geq  |f(2h)-2 f(h)+f(0)|= |-\frac{c}{\alpha^{\frac{1}{k}}}(2h)^{1+\alpha}+2 \frac{c}{\alpha^{\frac{1}{k}}}(h)^{1+\alpha}|=|\frac{c}{\alpha^{\frac{1}{k}}}h^{1+\alpha} (2^{\alpha}-1) |.
$$
Now since $0<\alpha<1$ is small, we have $(2^{\alpha}-1)\geq c' \alpha$. Thus,

$$
|\Delta^2_h f|_{L^{\infty}}\geq  |f(2h)-2 f(h))|\geq  c_1 \frac{\alpha}{\alpha^{\frac{1}{k}}}|h|^{1+\alpha},
$$
and hence we have 

$$
|f|_{B^{1}_{\infty,q}} \geq   c_1 \alpha^{1-\frac{1}{q}-\frac{1}{k}} .
$$

\end{proof}

\begin{lemma}\label{Type2}  Let  $0<\alpha<1$ and consider, for any $1\leq k \leq \infty$, the functions  $g_1,g_2:[0,1] \rightarrow \mathbb{R} $ defined as follows: 
$$g_1(r)=r^{1+\alpha} (\sin(\alpha^{1-\frac{1}{k}}\log(r)))^2, \,   \text{and} \,\, g_2(r)=r^{1+\alpha} \sin(\alpha^{1-\frac{1}{k}}\log(r))\,. $$
Then the Besov norm ${B^{1}_{\infty,q}}$  of $g_1$ and  $g_2$,  for $1 \leq q < \infty  $, satisfies
$$
  c_1 \alpha^{1-\frac{1}{q}-\frac{1}{k}}  \leq |g_1|_{B^{1}_{\infty,q}}, \, |g_2|_{B^{1}_{\infty,q}} \leq   c_0 \alpha^{1-\frac{1}{q}-\frac{1}{k}},
$$for $c_0$ and $c_1$ independent of $\alpha$.

\end{lemma}

\begin{proof}

The upper bounds follow similarly to the previous lemma. We will show how to obtain the lower bounds. It is clear that 
$$
|\Delta^2_h g_1|_{L^{\infty}}\geq  |g_1(2h)-2 g_1(h)|= | (2h)^{1+\alpha}  (\sin(\alpha^{1-\frac{1}{k}}\log(2h)))^2-2  (h)^{1+\alpha} (\sin(\alpha^{1-\frac{1}{k}}\log(h)) )^2|.$$
Now observe that we can write 
\begin{align}\label{Deltahf}
g_1(2h)-2 g_1(h)&= 2 h^{1+\alpha}  \bigg(2^{\alpha}(\sin(\alpha^{1-\frac{1}{k}}\log(2h))^2-  (\sin(\alpha^{1-\frac{1}{k}}\log(h))^2) \bigg).
 \end{align}
 Thus,
$$
\sin(\alpha^{1-\frac{1}{k}}\log(2h)= \sin(\alpha^{1-\frac{1}{k}}\log(h))\cos( \alpha^{1-\frac{1}{k}} \log(2))+\cos(\alpha^{1-\frac{1}{k}}\log(h) ) \sin(  \alpha^{1-\frac{1}{k}} \log(2)).
$$
Hence, 
\begin{align*}
(\sin(\alpha^{1-\frac{1}{k}}\log(2h))^2=
&\big(\sin(\alpha^{1-\frac{1}{k}}\log(h))\cos( \alpha^{1-\frac{1}{k}} \log(2))\big)^2+\\
 &2 \sin(\alpha^{1-\frac{1}{k}}\log(h))\cos( \alpha^{1-\frac{1}{k}} \log(2))\cos(\alpha^{1-\frac{1}{k}}\log(h) ) \sin(  \alpha^{1-\frac{1}{k}} \log(2)) +   \\
   &\big(\cos(\alpha^{1-\frac{1}{k}}\log(h) ) \sin(  \alpha^{1-\frac{1}{k}} \log(2))\big)^2.  
 \end{align*}
 
 Thus, when  $\alpha$  is small, the   leading term of $2^{\alpha}\big(\sin(\alpha^{1-\frac{1}{k}}\log(h))\cos( \alpha^{1-\frac{1}{k}} \log(2))\big)^2$, which is $ \big(\sin(\alpha^{1-\frac{1}{k}}\log(h))\big)^2$, cancels with $(\sin(\alpha^{1-\frac{1}{k}}\log(h))^2))$ in the right hand of \eqref{Deltahf}.  Therefore, we have 
 \begin{align*}
|g_1(2h)-2 g_1(h)|&\geq  2 |h|^{1+\alpha}  |2 \sin(\alpha^{1-\frac{1}{k}}\log(h))\cos( \alpha^{1-\frac{1}{k}} \log(2))\cos(\alpha^{1-\frac{1}{k}}\log(h) ) \sin(  \alpha^{1-\frac{1}{k}} \log(2))|.
 \end{align*}
 Hence, we have

\ \begin{align}\label{Deltahf2}
|g_1(2h)-2 g_1(h)|&\geq  c  \alpha^{1-\frac{1}{k}}   |h|^{1+\alpha}  | \sin(2\alpha^{1-\frac{1}{k}}\log(h)) | ,\end{align}
where $c$ is independent of $\alpha$. Therefore, we have

  $$
|g_1|^q_{B^{1}_{\infty,q}} \geq  
c   \alpha^{q-\frac{q}{k}}     \int_{0}^1  |h|^{q+q\alpha} |\sin(2 \alpha^{1-\frac{1}{k}}\log(h))|^q \frac{1}{|h|^{1+q}}   dh =c   \alpha^{q-\frac{q}{k}}    \int_{0}^1  |h|^{q\alpha} |\sin(2 \alpha^{1-\frac{1}{k}}\log(h))|^q \frac{1}{|h|}   dh.$$

Changing variables to $u=2 \alpha^{1-\frac{1}{k}}\log(h)$, we have

$$
 \alpha^{q-\frac{q}{k}}     \int_{0}^1  |h|^{q\alpha} |\sin(2 \alpha^{1-\frac{1}{k}}\log(h))|^q \frac{1}{|h|}   dh = \frac{ \alpha^{q-\frac{q}{k}} }{2 \alpha^{1-\frac{1}{k}}}   \int_{0}^{\infty}      e^{-\frac{q\alpha^{\frac{1}{k}} }{2 } u }  |\sin(u)|^q \, du.   
$$

Now observe that 

$$
   \int_{0}^{\infty}      e^{-\frac{q\alpha^{\frac{1}{k}} }{2 } u }  |\sin(u)|^q    du \geq \sum_{k=0}^{\infty}\int_{\frac{\pi}{4}+k\pi}^{\frac{3\pi}{4}+k\pi}      e^{-\frac{q\alpha^{\frac{1}{k}} }{2 } u }  |\sin(u)|^q    du \geq  \frac{1}{2^{q/2}}\sum_{j=0}^{\infty}\int_{\frac{\pi}{4}+j\pi}^{\frac{3\pi}{4}+j\pi}      e^{-\frac{q\alpha^{\frac{1}{k}} }{2 } u }     du = \frac{c}{2^{q/2}} \frac{1}{q\alpha^{\frac{1}{k}}}.
$$

Hence,

$$
|g_1|^q_{B^{1}_{\infty,q}} \geq   \alpha^{q-\frac{q}{k}}     \int_{0}^1  |h|^{q\alpha} |\sin(2 \alpha^{1-\frac{1}{k}}\log(h))|^q \frac{1}{|h|}   dh \geq    \frac{\alpha^{q-\frac{q}{k}} }{2 \alpha^{1-\frac{1}{k}}}    \frac{c}{2^{q/2}} \frac{1}{q\alpha^{\frac{1}{k}}}= c  \alpha^{q-\frac{q}{k}-1} . 
 $$

Thus, we obtain our desired lower bound

$$
|g_1|_{B^{1}_{\infty,q}} \geq   c  \alpha^{1-\frac{1}{k}-\frac{1}{q}}, 
$$

where $c$ is independent of $\alpha$. The estimate on $g_2$ follows similarly  from trig identities. 

\end{proof}

\begin{lemma} Let  $0<\alpha<1$,  $1\leq q<\infty$, and consider the following vorticity on $\mathbb{R}^2$:  
 \begin{align*}
\omega_A(r,\theta)&=\alpha^{\frac{1}{q}}r^{\alpha} \sin(2\theta) , \quad  \omega_{B}(r,\theta)= \alpha^{\frac{1}{q}}r^{\alpha}\cos(\alpha^{\frac{1}{q}}\log(r))\sin(2\theta), \\
\omega_{C}(r,\theta)&= r^{\alpha}\sin(\alpha^{\frac{1}{q}}\log(r)), \quad \omega_{D}(r,\theta)=r^{\alpha}\sin(\alpha^{\frac{1}{q}}\log(r)) \sin(2\theta).
\end{align*}

All compactly supported on the ball of radius 1. Then the velocities   generated by each of them have the corresponding   Besov norm ${B^{1}_{\infty,q}}$: 
$$
|u(\omega_A)|_{B^{1}_{\infty,q}} , |u(\omega_B)|_{B^{1}_{\infty,q}}, |u(\omega_B)|_{B^{1}_{\infty,q}} \leq c,     
$$
where  $c$ is independent  of $\alpha$. But we have
$$
|u(\omega_D)|_{B^{1}_{\infty,q}} \geq  \frac{c}{\alpha^{1-\frac{1}{q}}}.     
$$

\end{lemma}

\begin{proof}
First, let us  consider $\omega_A$ and $\omega_B$. There are couple of ways to see this.   Here, we have  $|\omega_A|_{C^{\alpha}}, |\omega_B|_{C^{\alpha}} \leq c\alpha^{\frac{1}{q}} $. Thus,  $$|\omega_A|_{B^{\alpha}_{\infty,\infty}}, |\omega_B|_{B^{\alpha}_{\infty,\infty}} \leq c\alpha^{\frac{1}{q}},$$ which implies that $$|\omega_A|_{B^{0}_{\infty,q}}, |\omega_B|_{B^{0}_{\infty,q}} \leq c, $$
where  $c$ is independent  of $\alpha$. This can be seen by using Littlewood-Paley theory. Since  the Riesz operators are bounded on   Besov spaces, we have   $$
|u(\omega_A)|_{B^{1}_{\infty,q}} , |u(\omega_B)|_{B^{1}_{\infty,q}}  \leq c.      
$$

One can also see this by computing the velocity field generated by $\omega_A$ and $\omega_B$. Since the vorticity has $\sin(2\theta)$ symmetry, the stream functions will also be $\sin(2\theta)$ symmetric. Thus, one can compute the velocity field by solving $\Delta \psi = \omega$; this reduces to an ODE, which we can explicitly solve. For vorticity $\omega_{A}$, one observes that the main term in the velocity field (by the Biot-Svart decomposition, Proposition \ref{ElgindiCcirc}), for instance for the first component $u_1$, will be:
\begin{equation}\label{omegaAtype1}
u_1(\omega_A)=rL^s_{12}(\omega_A)\sin(\theta)=\frac{c}{\alpha^{1-\frac{1}{q}}}r (1- r^{\alpha})\sin(\theta).\end{equation}

This can also be computed it directly through the ODE. We observe that \eqref{omegaAtype1} is of the same form as Lemma \ref{Type1}. Similarly, for $\omega_{B}$, we note that  the main term in the velocity field, by the Biot-Svart decomposition Proposition \ref{ElgindiCcirc}, is the following: 

\begin{equation}\label{omegaBtype2}
u_1(\omega_B)=rL^s_{12}(\omega_B)\sin(\theta)=cr^{1+\alpha}\sin(\alpha^{\frac{1}{q}}\log(r))\sin(\theta). 
\end{equation}
 Hence, we observe it is of the same form as in  Lemma \ref{Type2}.  Now for vorticity $\omega_C$, since it is radial, we can also compute the velocity field through the stream function, and see that the main term in the velocity field will be: 
$$
u_1(\omega_C) =cr^{1+\alpha}\sin(\alpha^{\frac{1}{q}}\log(r))\sin(\theta). 
$$
Therefore, it is of the same form as $u_1(\omega_B)$. Thus, we have  $$|u(\omega_C)|_{B^{1}_{\infty,q}} \leq c. $$ Finally, for vorticity $\omega_D$,  we can  compute the main term in the velocity field similarly, and it will be of the following form: 

$$
u(\omega_D) =rL^s_{12}(\omega_D)\sin(\theta)=\frac{c}{\alpha^{1-\frac{1}{q}}}r^{1+\alpha}\cos(\alpha^{\frac{1}{q}}\log(r))\sin(\theta). 
$$
Thus, from  Lemma \ref{Type1} and  Lemma \ref{Type2},  by rescaling, we have 

$$
|u(\omega_D)|_{B^{1}_{\infty,q}}\geq \frac{c}{\alpha^{1-\frac{1}{q}}},  
$$
and this completes the proof.

 \end{proof}

\section{Main Estimate for the Leading Order Model  }\label{LOME}
In this section, we prove the main estimates for the leading order model. In Lemma \ref{UpperBndL12}, we establish an upper bound on the gradient of velocity of the leading order model, which we will use in Proposition \ref{LOMLowerBndL12} to control the dynamics and prove a lower bound on the the  gradient of velocity of the leading order model. Then, we use Lemma \ref{UpperBndL12} and  Proposition \ref{LOMLowerBndL12}    in Proposition \ref{L12cBesovLowBnd}    to prove the ill-posedness in  $B^{1}_{\infty,q}$ of the leading order model, for $1<q<\infty$.

\begin{lemma}\label{UpperBndL12} For any $1<q<\infty$, let $\Omega(t)$ be a solution to the leading order model \eqref{LOMeqExp} with initial data 
\begin{equation} \label{Omega0UpperBndL12}
\Omega_0(r,\theta) =r^{\alpha} \sin(\alpha^{\frac{1}{q}}\log(r))+\alpha^{\frac{1}{q}} r ^{\alpha} \sin(2\theta)
\end{equation} 

compactly supported on the ball of radius 1, then with $L^s_{12}(\Omega)$  defined as in \eqref{L12s},   there exist constants $c$ independent of $\alpha$ such that we have the following estimates 
 
\begin{equation} \label{L12Up}
 |L^s_{12}(\Omega)(t)|_{L^{\infty}} \leq \frac{c}{\alpha^{1-\frac{1}{q}}} ,
 \end{equation} 

on the time scale  $0\leq t \leq T(\alpha)=c\alpha^{1-\frac{1}{q}}\log(c|\log(\alpha)|) $, with c independent of $\alpha$.

\end{lemma}

\begin{proof}

We will prove this proposition using a bootstrap argument. Namely, we assume that there is constant a $c$ such that 
\begin{equation}\label{BootstrapH}
  |L^s_{12}(\Omega)(t)|_{L^{\infty}} \leq   \frac{10c}{\alpha^{1-\frac{1}{q}}},  
\end{equation} 
 and we will show that 
 \begin{equation}\label{BootstrapC}
  |L^s_{12}(\Omega)(t)|_{L^{\infty}} \leq   \frac{c}{\alpha^{1-\frac{1}{q}}}.  
\end{equation} 
We consider the  leading order model:  
$$
 \dt \Omega+  2 L^s_{12}(\Omega) \sin(2\theta)  \dth \Omega- 2 rL^s_{12}(\Omega) \cos(2\theta) \dr \Omega =0.
 $$
By studying the characteristics of the transport equation, we can write the coupled system of ODEs for the radial and angular particle trajectories for the flow map.  Only for now, to shorten the notation, we write $\Phi(t,r,\theta)=(\Phi_r (t,r,\theta),\Phi_{\theta}(r,\theta)))=(r(t),\theta(t))$, where  when $t=0$, we have   $\Phi( 0,r,\theta) =(r(0),\theta(0))=(r,\theta)$. Thus, the flow maps satisfies the following:  
 \begin{align}
\dot{r}(t)&=  - 2r(t) L^s_{12}(\Omega) \cos(2\theta(t))\label{RadialTraj}, \\
\dot{\theta}(t)&=2 L^s_{12}(\Omega)  \sin(2\theta(t))\label{AnguarTraj}.
\end{align}
Therefore, we can write 
 $$
\frac{dr }{d\theta }=-\frac{  r(t)   \cos(2\theta(t))}{     \sin(2\theta(t))}.
$$
Hence, we obtain 

 $$
 r(t) |\sin(2\theta(t))|^{\frac{1}{2}}  =     r   |\sin(2\theta)|^{\frac{1}{2}}, 
 $$
 which can be written as follows: 
 
\begin{equation}\label{FlowMapEq}
 (r(t))^2 |\sin(2\theta(t))|   =     r^2  |\sin(2\theta)| .
\end{equation} 
This corresponds to the trajectories of a hyperbolic flow. Now, since the flow map trajectories satisfy \eqref{FlowMapEq}, the inverse flow also satisfies \eqref{FlowMapEq}, with the trajectories flowing in the opposite direction. Thus, from now onward, with a slight abuse of notation, we will write $\Phi^{-1}(t,r,\theta)=(\Phi^{-1}_r (t,r,\theta),\Phi^{-1}_{\theta}(r,\theta)))=(r(t),\theta(t))$, where  when $t=0$, we have   $\Phi^{-1}( 0,r,\theta) =(r(0),\theta(0))=(r,\theta)$. Recall from $\eqref{Omega0UpperBndL12}$, that initial data is the following:  $$\Omega_0(r,\theta) =r^{\alpha} \sin(\alpha^{\frac{1}{q}}\log(r))+\alpha^{\frac{1}{q}} r ^{\alpha} \sin(2\theta).$$
 With $(\Phi^{-1}(t,r,\theta))=(r(t),\theta(t))$, we have 
 $$
 \Omega(t,r,\theta)= \Omega_0(\Phi^{-1}(t,r,\theta))=\Omega_0(r(t),\theta(t)) =(r(t))^{\alpha} \sin(\alpha^{\frac{1}{q}}\log(r(t)))+\alpha^{\frac{1}{q}} (r (t) )^{\alpha} \sin(2\theta(t)).
 $$
 From \eqref{FlowMapEq} and \eqref{RadialTraj}, we can write 
$$
r(t)=r\exp{( \int_0^{t} 2  L^s_{12}(r(\tau)) \cos(2\theta(\tau))) d\tau}.
$$
To shorten the notation, we introduce 
 $$
  X(t,r,\theta):=2\int_0^t(L_{12}(r(\tau))\cos(2\theta(\tau)))d\tau \implies \,\, \text{we can write} \,\,  r(t)=re^{X(t)}. 
  $$
Thus, we have   
\begin{equation}\label{OmegaT}
 \Omega(t,r,\theta)=    r^{\alpha}e^{\alpha X(t)} \sin(\alpha^{\frac{1}{q}}\log(re^{X(t)}))+\alpha^{\frac{1}{q}}  r^{\alpha}e^{\alpha X(t)} \sin(2\theta(t)).
\end{equation} 
  Recall that 
 
  $$
 L^{s}_{12} (\Omega)(r)=c \int_r^{\infty} \int_0^{2\pi}  \frac{ 1}{s} \Omega(t,s,\theta) \sin(2\theta) d \theta \, ds.
$$
Using symmetry, it suffices to consider the $\theta$ integral on $[0,\frac{\pi}{2}]$, and since our data has compact support, the operator can be written as follows:

$$
 L^{s}_{12} (\Omega)(r)=c \int_r^{1} \int_0^{\frac{\pi}{2}}  \frac{ 1}{s} \Omega(t,s,\theta) \sin(2\theta) d \theta \, ds
$$
Using \eqref{OmegaT}, we obtain 
\begin{equation}\label{L12splitI1I2}
  L_{12}^{s}(\Omega)=c\int_{r}^{1}\int_{0}^{\frac{\pi}{2}}\frac{ 1}{s} \Big(  s^{\alpha}e^{\alpha X(t)} \sin(\alpha^{\frac{1}{q}}\log(se^{X(t)}))+\alpha^{\frac{1}{q}}  s^{\alpha}e^{\alpha X(t)} \sin(2\theta(t)) \Big)    \sin(2\theta) d\theta ds:= I_1+I_2,
\end{equation} 
where the terms $I_1$ and $I_2$  are defined as follows:
 \begin{align}
I_1&= c\int_{r}^{1}\int_{0}^{\frac{\pi}{2}}\frac{ 1}{s} \Big(  s^{\alpha}e^{\alpha X(t)} \sin(\alpha^{\frac{1}{q}}\log(se^{X(t)})) \Big)    \sin(2\theta) d\theta ds, \\
I_2&=c \int_{r}^{1}\int_{0}^{\frac{\pi}{2}}\frac{ 1}{s} \Big(   \alpha^{\frac{1}{q}}  s^{\alpha}e^{\alpha X(t)} \sin(2\theta(t)) \Big)    \sin(2\theta) d\theta ds.
\end{align}
 For $I_1$, we observe that  $X(t)=2\int_0^t(L_{12}(s(\tau))\cos(2\theta(\tau)))d\tau$   is even in $\theta$. Thus, we have 

\begin{equation}\label{I1Est}
I_1= c\int_{r}^{1}\int_{0}^{\frac{\pi}{2}}\frac{ 1}{s} \Big(  s^{\alpha}e^{\alpha X(t)} \sin(\alpha^{\frac{1}{q}}\log(se^{X(t)})) \Big)    \sin(2\theta) d\theta ds=0
\end{equation} 
 For $I_2$, we  will use the bootstrap assumption. Namely,  from \eqref{BootstrapH}, we have 
$$
  |L^s_{12}(\Omega)(t)|_{L^{\infty}} \leq   \frac{10c}{\alpha^{1-\frac{1}{q}}}.  
$$
Therefore,  
$$
|X(t)|_{L^{\infty}} \leq  \frac{10c}{\alpha^{1-\frac{1}{q}}} t \implies \alpha |X(t)|_{L^{\infty}} < \alpha^{\frac{1}{q}}     10c  t <1,
$$
on the time scale $0\leq t \leq T(\alpha)=c \alpha^{1-\frac{1}{q}}\log(c|\log(\alpha)|) $. Thus, we have $ |e^{\alpha X(t)}|<2, $  which gives

\begin{equation}\label{I2Est}
|I_2| \leq \int_{r}^{1}\int_{0}^{\frac{\pi}{2}}\frac{ 1}{s}   \alpha^{\frac{1}{q}}  s^{\alpha} |e^{\alpha X(t)}| \, |\sin(2\theta(t))|    |\sin(2\theta)| d\theta ds \leq  c \int_{r}^{1}\int_{0}^{\frac{\pi}{2}}\frac{  \alpha^{\frac{1}{q}}  s^{\alpha}}{s}      d\theta ds \leq \frac{c}{\alpha^{1-\frac{1}{q}}}.
\end{equation}
 Hence, from \eqref{I2Est}, \eqref{I1Est},  \eqref{L12splitI1I2}, and by adjusting the constant $c$, if necessary, we have 
\begin{equation} 
 | L_{12}^{s}(\Omega)|_{L^{\infty}} \leq  \frac{c}{\alpha^{1-\frac{1}{q}}} , 
\end{equation} 
which completes the bootstrap argument, and proves the lemma.

\end{proof}

\begin{proposition}\label{LowerBndL12} For any $1<q<\infty$, let $\Omega(t)$ be a solution to the leading order model:

\begin{equation}\label{LOMLowerBndL12} 
 \dt \Omega+  2 L^s_{12}(\Omega) \sin(2\theta)  \dth \Omega- 2 rL^s_{12}(\Omega) \cos(2\theta) \dr \Omega =0
\end{equation} 

 with initial data 
\begin{equation} \label{Omega0LowerBndL12}
\Omega_0(r,\theta) =r^{\alpha} \sin(\alpha^{\frac{1}{q}}\log(r))+\alpha^{\frac{1}{q}} r ^{\alpha} \sin(2\theta)
\end{equation}

compactly supported on the ball of radius 1, then with $L^s_{12}$  defined as in \eqref{L12s},   there exist constants $c$ independent of $\alpha$ such that we have the following estimates 
 
\begin{equation} \label{IntL12Down}
\sup_{0\leq r \leq 1}| \int_0^{t}L^s_{12}(\Omega) d\tau|\geq  \log(1+c\frac{t}{\alpha^{1-\frac{1}{q}}})
\end{equation} 
on the time scale  $0\leq t \leq T(\alpha)=c\alpha^{1-\frac{1}{q}}\log(c|\log(\alpha)|) $, with c independent of $\alpha$.

\end{proposition}

\begin{remark}
Similar type of estimate was obtained in our previous work with Elgindi \cite{ESK}. 
\end{remark}

\begin{proof}
We proceed in a similar manner as in Lemma  \ref{UpperBndL12}. Recall by setting  
$$\Phi^{-1}(t,r,\theta)=(\Phi^{-1}_r (t,r,\theta),\Phi^{-1}_{\theta}(r,\theta)))=(r(t),\theta(t)),$$ 
where  when $t=0$, we have   $\Phi^{-1}( 0,r,\theta) =(r(0),\theta(0))=(r,\theta)$, and by studying the characteristics of the leading order model, we obtained:
 
 \begin{equation}\label{FlowMapEqLBnd}
 (r(t))^2 |\sin(2\theta(t))|   =     r^2  |\sin(2\theta)|. 
\end{equation} 
Using the characteristics  of \eqref{LOMLowerBndL12} and equation \eqref{FlowMapEqLBnd}, we can write 

$$
r(t)=r\exp{( \int_0^{t} 2  L^s_{12}(\Omega(\tau)) \cos(2\theta(\tau)))}
$$
Similarly, as in  Lemma \ref{UpperBndL12}, shortening  the notation, we   write  $r(t)=re^{X(t)}$, where

 \begin{equation}\label{Xflow}
  X(t,r,\theta)=2\int_0^t(L^{s}_{12}(r(\tau))\cos(2\theta(\tau))).
\end{equation} 
Hence, we can write the solution as follows:  
 $$
 \Omega(t,r,\theta)= \Omega_0(\Phi^{-1}(t,r,\theta))  = r^{\alpha}e^{\alpha X(t)} \sin(\alpha^{\frac{1}{q}}\log(re^{X(t)}))+\alpha^{\frac{1}{q}}  r^{\alpha}e^{\alpha X(t)} \sin(2\theta(t)). 
 $$
 As in Lemma \ref{UpperBndL12}, we compute the operator $L^s_{12}$. Because of symmetry and the compact support of the initial data, it suffices to consider  
      $$
  L_{12}^{s}(\Omega)=c\int_{r}^{1}\int_{0}^{\frac{\pi}{2}}\frac{ 1}{s} \Big(  s^{\alpha}e^{\alpha X(t)} \sin(\alpha^{\frac{1}{q}}\log(se^{X(t)}))+\alpha^{\frac{1}{q}}  s^{\alpha}e^{\alpha X(t)} \sin(2\theta(t)) \Big)    \sin(2\theta) d\theta ds. 
  $$
Now from $\eqref{Xflow}$, and since we are integrating against $\sin(2\theta)$ term, again using symmetry, we have 
  $$
 \int_{r}^{1}\int_{0}^{\frac{\pi}{2}}\frac{ 1}{s} \Big(  s^{\alpha}e^{\alpha X(t)} \sin(\alpha^{\frac{1}{q}}\log(se^{X(t)})) \Big)    \sin(2\theta) d\theta ds =0. 
  $$
  Thus, we obtain the following: 
  
 \begin{equation}\label{L12smp0}
  L_{12}^{s}(\Omega)(t)=c\int_{r}^{1}\int_{0}^{\frac{\pi}{2}}\frac{ 1}{s} \Big( \alpha^{\frac{1}{q}}  s^{\alpha}e^{\alpha X(t)} \sin(2\theta(t)) \Big)    \sin(2\theta) d\theta ds.
 \end{equation} 
  Using \eqref{FlowMapEqLBnd}, we can rewrite $ L_{12}^{s}(\Omega)$ as follows

 \begin{equation}\label{L12smp}
     L_{12}^{s}(\Omega)(t)=\int_{r}^{1}\int_{0}^{\frac{\pi}{2}}\frac{ \alpha^{\frac{1}{q}}  e^{(\alpha-1) X(t)}  s^{\alpha} }{s}             \sin^2(2\theta) d\theta ds. 
 \end{equation} 
Therefore, taking the time derivative, we obtain 
  $$
    \partial_t     L_{12}^{s}(\Omega)(t)=\int_{r}^{1}\int_{0}^{\frac{\pi}{2}}\frac{ \alpha^{\frac{1}{q}}  e^{(\alpha-1) X(t)}  s^{\alpha} }{s}       (\alpha-1) X'(t)       \sin^2(2\theta) d\theta ds.  
  $$
 Using $\eqref{Xflow}$ gives     
  
  $$
    \partial_t     L_{12}^{s}(\Omega)(t)=c\int_{r}^{1}\int_{0}^{\frac{\pi}{2}}\frac{ \alpha^{\frac{1}{q}}  e^{(\alpha-1) X(t)}  s^{\alpha} }{s}       (\alpha-1)  \Big( 2  L^s_{12}(\Omega(\tau)) \cos(2\theta(\tau) \Big)      \sin^2(2\theta) d\theta ds. 
  $$
 Thus, we have the following

 \begin{equation}\label{dtL12smp}
    \partial_t     L_{12}^{s}(\Omega)(t)\geq - c \int_{r}^{1}\int_{0}^{\frac{\pi}{2}}\frac{ \alpha^{\frac{1}{q}}  e^{(\alpha-1) X(t)}  s^{\alpha} }{s}        2 L^s_{12}(\Omega(t))        \sin^2(2\theta) d\theta ds. 
 \end{equation} 
From $\eqref{L12smp}$, we can take the radial derivative to obtain  
 \begin{equation}\label{dsL12smp}
    \dr L_{12}^{s}(\Omega)(t)= - c\int_{0}^{2\pi}\frac{ \alpha^{\frac{1}{q}}  e^{(\alpha-1) X(t)}  s^{\alpha} }{s}   \sin^2(2\theta) d\theta ds. 
 \end{equation}
  Hence, from \eqref{dtL12smp} and \eqref{dtL12smp} we get 
     $$
   \partial_t  L_{12}^{s}(\Omega)(t)\geq  \int_{r}^{1} 2    \partial_{s} L_{12}^{s}(\Omega)  L^s_{12}(\Omega(\tau)) = \int_{r}^{1}    \partial_{s} \Big ( L_{12}^{s}(\Omega)  \Big)^2         ds  .          
   $$
Thus, we  obtain the following Riccati-type equation
    \begin{equation}\label{dtL12Riccat}
   \partial_t  L_{12}^{s}(\Omega)(t)\geq -    (L_{12}^{s}(\Omega)(t) )^2.        
    \end{equation} 
Therefore,  
  $$
  L_{12}^{s}(\Omega)(t) \geq \frac{L^s_{12}(\Omega)(0)}{1+tL^s_{12}(\Omega)(0)},
  $$
which gives 
   $$
  \int_0^{t}L_{12}^{s}(\Omega)(\tau)  \geq   \log(1+tL^s_{12}(\Omega)(0)). 
  $$
 Hence, we have our result.

   \end{proof}

   \begin{proposition} \label{L12cBesovLowBnd} 
   
For any $1<q<\infty$, let $\Omega(t)$ be a solution to the leading order model:

\begin{equation}  \label{LOMBesov} 
 \dt \Omega+  2 L^s_{12}(\Omega) \sin(2\theta)  \dth \Omega- 2 rL^s_{12}(\Omega) \cos(2\theta) \dr \Omega =0
\end{equation} 

 with initial data 
\begin{equation*}
\Omega_0(r,\theta) =r^{\alpha} \sin(\alpha^{\frac{1}{q}}\log(r))+\alpha^{\frac{1}{q}} r ^{\alpha} \sin(2\theta)
\end{equation*}

compactly supported on the ball of radius 1, then with $L^c_{12}$  defined as in \eqref{L12c},   there exist constants $c$ independent of $\alpha$ such that we have the following estimates 
 
\begin{equation} \label{rL12cBesovLower}
|r L_{12}^{c}(\Omega)(t,r)|_{B^{1}_{{\infty},q}}   \geq     c     \log(c\frac{t}{\alpha^{1-\frac{1}{q}}}+1).  
\end{equation} 
on the time scale  $0\leq t \leq T(\alpha)=c\alpha^{1-\frac{1}{q}}\log(c|\log(\alpha)|) $, with c independent of $\alpha$. Thus,  at time $T(\alpha)=\alpha^{1-\frac{1}{q}} \log(c|\log(\alpha)|) \rightarrow 0 \, \text{, as} \,\, \alpha \rightarrow 0$, we have 
 
\begin{equation} \label{rL12cBesovLower}
|r L_{12}^{c}(\Omega)|_{B^{1}_{{\infty},q}}   \geq     c     \log(c |\log(\alpha)|)  \rightarrow \infty. 
\end{equation}

   \end{proposition}

   \begin{proof}
   
   We begin in a similar manner as in Lemma \ref{UpperBndL12} and Proposition \ref{LowerBndL12}. We set, $\Phi^{-1}(t,r,\theta)=(\Phi^{-1}_r (t,r,\theta),\Phi^{-1}_{\theta}(r,\theta)))=(r(t),\theta(t))$. Recall from  Lemma \ref{UpperBndL12}, by studying the characteristic of the leading order model, we obtained  
 \begin{equation}\label{FlowMapEqBesov}
 (r(t))^2 |\sin(2\theta(t))|   =     r^2  |\sin(2\theta)| .
\end{equation} 
In addition, as in  Lemma \ref{UpperBndL12},  using the characteristic  of the leading order model \eqref{LOMBesov} and equation \eqref{FlowMapEqBesov}, shortening  the notation, we write  $r(t)=re^{X(t)}$, where  
 \begin{equation}\label{XflowBesov}
  X(t,r,\theta)=2\int_0^t(L_{12}(r(\tau))\cos(2\theta(\tau))) d\tau.
\end{equation} 
From \eqref{FlowMapEqBesov},  we can solve for $\cos(2\theta(t))$ and obtain: 
 \begin{equation}\label{cos2thetaplus}
  \cos(2\theta(t))   =  \sqrt{1 - \frac{r^4 (\sin^2(2\theta)))}{r(t)^4 } },  \,\,   \text{on}  \,\,   -\frac{\pi}{4}\leq \theta \leq \frac{\pi}{4}  \,\,   \text{, and on}      \,\,   \frac{3\pi}{4}\leq \theta \leq \frac{5\pi}{4}, 
 \end{equation} 
 
 and 
 \begin{equation}\label{cos2thetaneg}
  \cos(2\theta(t))   = - \sqrt{1 - \frac{r^4 (\sin^2(2\theta)))}{r(t)^4 } },  \,\,   \text{on}  \,\,  \frac{\pi}{4}\leq \theta \leq \frac{3\pi}{4}  \,\,   \text{, and on}      \,\,  \frac{5\pi}{4}\leq \theta \leq \frac{7\pi}{4}.
 \end{equation} 
 We consider the operator $L^{c}_{12}$.  Recall from \eqref{L12c} that $L^{c}_{12}$ is defined as follows:
$$
 L^{c}_{12} (\Omega)(r)=c \int_r^{1} \int_0^{\frac{\pi}{2}}  \frac{ 1}{s} \Omega(t,s,\theta) \cos(2\theta) d \theta \, ds. 
$$
Here, we used symmetry and compact support of initial data. We know that we can write the solution as follows: 
 
 $$
 \Omega(t,r,\theta)=  r^{\alpha}e^{\alpha X(t)} \sin(\alpha^{\frac{1}{q}}\log(re^{X(t)}))+\alpha^{\frac{1}{q}}  r^{\alpha}e^{\alpha X(t)} \sin(2\theta(t)).
 $$
Therefore, we can compute $L^{c}_{12}$,
  
  $$
 L^{c}_{12} (\Omega)(r)=c\int_{r}^{1}\int_{0}^{\frac{\pi}{2}}\frac{ 1}{s} \Big(  s^{\alpha}e^{\alpha X(t)} \sin(\alpha^{\frac{1}{q}}\log(se^{X(t)}))+\alpha^{\frac{1}{q}} s^{\alpha}e^{\alpha X(t)} \sin(2\theta(t)) \Big)    \cos(2\theta) d\theta ds
  $$
  First, we will simplify the expression. Due to symmetry, we have   
   $$
  \int_{r}^{1}\int_{0}^{2\pi}\frac{1}{s} \Big(  \alpha^{\frac{1}{q}}  s^{\alpha}e^{\alpha X(t)} \sin(2\theta(t)) \Big) \cos(2\theta) d\theta ds=0
  $$
  Thus, we have 
  $$
   L_{12}^{c}(\Omega)=c\int_{r}^{1}\int_{0}^{\frac{\pi}{2}}\frac{1}{s}    s^{\alpha}e^{\alpha X(t)} \sin(\alpha^{\frac{1}{q}}\log(s)+  \alpha^{\frac{1}{q}} X(t))   \cos(2\theta) d\theta ds, 
  $$
 which can be rewritten as follows:
   \begin{equation}\label{L12cPreExp}
   L_{12}^{c}(\Omega)=c\int_{r}^{1}\int_{0}^{\frac{\pi}{2}}\frac{1}{s}    s^{\alpha}e^{\alpha X(t)} \Big( \sin(\alpha^{\frac{1}{q}}\log(s)) \cos(\alpha^{\frac{1}{q}} X(t)) + \cos(\alpha^{\frac{1}{q}}\log(s))  \sin(\alpha^{\frac{1}{q}} X(t))  \Big)  \cos(2\theta) d\theta ds.
 \end{equation} 
   From Lemma \ref{UpperBndL12}, we have that  $$ |L^s_{12}(\Omega)(t)|_{L^{\infty}} \leq \frac{c}{\alpha^{1-\frac{1}{q}}}, $$ on the time scale    $0\leq t \leq T(\alpha)=c\alpha^{1-\frac{1}{q}}\log(c|\log(\alpha)|) $. Hence, from \eqref{XflowBesov}, on the same time scale, we have 
 
   $$|X(t)|_{L^{\infty}} \leq c \frac{t}{\alpha^{1-\frac{1}{q}}}. $$
 Therefore, by taking $\alpha$ sufficiently small, we   have    $$\alpha |X(t)|, \alpha^{\frac{1}{q}} |X(t)| < 1.$$ Thus, we can expand the following terms as follows:
\begin{align}
\cos(\alpha^{\frac{1}{q}} X(t))) &= 1- \alpha^{\frac{2}{q}} (X(t))^2 + \dots  \label{ExpC} \\
\sin(\alpha^{\frac{1}{q}} X(t)))&= \alpha^{\frac{1}{q}} X(t) +  \alpha^{\frac{3}{q}}  (X(t))^3 + \dots  \label{ExpS}\\
e^{\alpha X}&=1+ \alpha X  + \dots \label{ExpE}  
\end{align}
 Hence, from \eqref{L12cPreExp}  using \eqref{ExpC}, \eqref{ExpS}, and \eqref{ExpE},  we have 
$$
   L_{12}^{c}(\Omega)=c\int_{r}^{1}\int_{0}^{\frac{\pi}{2}}\frac{s^{\alpha}}{s} \Big( \sin(\alpha^{\frac{1}{q}}\log(s))   + \cos(\alpha^{\frac{1}{q}}\log(s) )  \alpha^{\frac{1}{q}} X(t)   + \alpha^{\frac{2}{q}}  X(t)E(s,X)  \Big)  \cos(2\theta) d\theta ds.
  $$
    The third term which has $E(s,X)$, which is bounded $|E(s,X)|_{\infty} \leq c$, contains all the remaining terms from the expansion. This term is an error term, which we will soon show is controlled.  
   Here, we observe that since the first term is radial, we have  
   $$
  \int_{r}^{1}\int_{0}^{\frac{\pi}{2}}\frac{s^{\alpha}}{s} \Big( \sin(\alpha^{\frac{1}{q}}\log(s))    \Big)  \cos(2\theta) d\theta ds=0.
  $$
  Hence, we obtain      
\begin{equation} \label{L12cExp2}
   L_{12}^{c}(\Omega)=c\int_{r}^{1}\int_{0}^{\frac{\pi}{2}} \Big(  \frac{s^{\alpha}}{s} \cos(\alpha^{\frac{1}{q}}\log(s) )  \alpha^{\frac{1}{q}} X(t)   + \alpha^{\frac{2}{q}} \frac{s^{\alpha} X(t) E(s,X)}{s}  \Big)   \cos(2\theta) d\theta ds.
\end{equation}
Note that the goal of the proposition is to show that $|r L_{12}^{c}(\Omega)|_{B^{1}_{\infty, q}}$  grows arbitrarily large on our time scale $T(\alpha)=c \alpha^{1-\frac{1}{q}}\log(c|\log(\alpha)|).$ Furthermore, recall that the second finite difference is defined as follows: 
$$
\Delta^2_h f(x)= f(x+2h)-2 f(x+h)+f(x), \,\, \text{and} \,\,  |\Delta^2_h f(x)|_{L^{\infty}}=\sup_{x}|f(x+2h)-2 f(x+h)+f(x)|.
$$
Therefore, we have
$$ | \Delta^2_h(r L_{12}^{c}(\Omega)(r))|_{L^{\infty}} \geq |2h L_{12}^{c}(\Omega)(2h)-2hL_{12}^{c}(\Omega)(h) |= | 2h \big( L_{12}^{c}(\Omega)(2h)- L_{12}^{c}(\Omega)(h) \big)  |.$$ 
 Hence, this can be written as:

 $$
   | \Delta^2_h(r L_{12}^{c}(\Omega)(r))|_{L^{\infty}} \geq \Big|c 2h \int_{h}^{2h} \int_{0}^{\frac{\pi}{2}} \Big(  \frac{r^{\alpha}}{r} \cos(\alpha^{\frac{1}{q}}\log(r) )  \alpha^{\frac{1}{q}} X(t)   +   \alpha^{\frac{2}{q}} \frac{r^{\alpha} X(t) E(r,X)}{r}  \Big)  \cos(2\theta) d\theta dr \Big|.
  $$
 We now note that since the second term has an $\alpha^{\frac{2}{q}}$ in front, it will be of size $\alpha^{\frac{1}{q}}|X(t)|_{L^{\infty}} $ in $B^{1}_{\infty,q}$. Since $\alpha^{\frac{1}{q}}|X(t)|_{L^{\infty}}\leq c \alpha^{\frac{1}{q}}|\log(\alpha)|$, on our time scale,  the error term will be of size $\alpha^{\frac{1}{q}}|\log(\alpha)|$ in  $B^{1}_{\infty,q}$ which goes to zero as $\alpha \rightarrow 0 $. Hence,  the second term (error term) is controlled and can be absorbed. Therefore, to shorten the notation, we will drop the second term and just write 
   $$
   | \Delta^2_h(r L_{12}^{c}(\Omega)(r))|_{L^{\infty}} \geq \Big|c 2h \int_{h}^{2h} \int_{0}^{\frac{\pi}{2}}\frac{r^{\alpha}}{r}  \cos(\alpha^{\frac{1}{q}}\log(r) )  \alpha^{\frac{1}{q}} X(t)       \cos(2\theta) d\theta dr \Big|.
  $$
  Recall from \eqref{XflowBesov}, we have 
  
  $$
   | \Delta^2_h(r L_{12}^{c}(\Omega)(r))|_{L^{\infty}} \geq \Big|2h \int_{h}^{2h} \int_{0}^{\frac{\pi}{2}}  \frac{r^{\alpha}}{r}  \cos(\alpha^{\frac{1}{q}}\log(r) )  \alpha^{\frac{1}{q}}  \int_0^t  2 L^{s}_{12}(\Omega(\tau,r))\cos(2\theta(\tau))  d\tau   \cos(2\theta) d\theta   dr \Big|.
  $$
First, we will focus on  the radial variable.  Changing variables by setting $u=\alpha^{\frac{1}{q}}\log(r)$, we have $  r=e^{ \frac{u}{\alpha^{\frac{1}{q}}}} $. Hence, the integral becomes

  $$
  \Big|2h  \int_{\alpha^{\frac{1}{q}}\log(h)}^{\alpha^{\frac{1}{q}}\log(2h)}  \int_{0}^{\frac{\pi}{2}}   e^{\alpha^{1-\frac{1}{q}}u}  \cos(u )      \int_0^t  2 L^{s}_{12}(\Omega(\tau,e^{ \frac{u}{\alpha^{\frac{1}{q}}}}))\cos(2\theta(\tau))  d\tau   \cos(2\theta) d\theta   du \Big|.
  $$
We observe that this can be rewritten as follows:

  $$
  \Big|\frac{\alpha^{\frac{1}{q}}\log(2)}{\alpha^{\frac{1}{q}}\log(2)}2h  \int_{\alpha^{\frac{1}{q}}\log(h)}^{\alpha^{\frac{1}{q}}\log(h)+ \alpha^{\frac{1}{q}}\log(2)}   \int_{0}^{\frac{\pi}{2}}   e^{\alpha^{1-\frac{1}{q}}u}  \cos(u )      \int_0^t  2 L^{s}_{12}(\Omega(\tau,e^{ \frac{u}{\alpha^{\frac{1}{q}}}}))\cos(2\theta(\tau))  d\tau   \cos(2\theta) d\theta   du \Big|
  $$
  Using Lebesgue differentiation theorem, by  taking $\alpha$ sufficiently small, we obtain

\begin{align}\label{L12dThetadt}
   | \Delta^2_h(r L_{12}^{c}(\Omega)(r))|_{L^{\infty}} \geq  c  \Big|\alpha^{\frac{1}{q}}  2h    \int_{0}^{\frac{\pi}{2}}    h^{\alpha}\cos(\alpha^{\frac{1}{q}}\log(h))      \int_0^t  2 L^{s}_{12}(\Omega(\tau,h))\cos(2\theta(\tau))  d\tau   \cos(2\theta) d\theta     \Big|  
  \end{align}
 Second, we will now focus on the angular integral. Recall from  \eqref{cos2thetaplus} and \eqref{cos2thetaneg}, we have

$$
 \cos(2\theta(\tau))   =  \sqrt{1 - \frac{r^4 (\sin^2(2\theta)))}{r(\tau)^4 } }, \,\,   \text{on}  \,\,  0 \leq \theta \leq \frac{\pi}{4}, \, \text{and} \,\, \cos(2\theta(\tau))   = - \sqrt{1 - \frac{r^4 (\sin^2(2\theta)))}{r(\tau)^4 } } \,\,   \text{on}  \,\,    \frac{\pi}{4}\leq \theta \leq \frac{\pi}{2}.
 $$
Thus, observing the $\theta$ integral in $\eqref{L12dThetadt}$, we have

\begin{equation} \label{ThetaInt}
\int_{0}^{\frac{\pi}{2}}   \cos(2\theta(\tau))     \cos(2\theta) d\theta  =   \int_{0}^{\frac{\pi}{4}}   \sqrt{1 - \frac{s^4  \sin^2(2\theta)}{s(\tau)^4 } } \cos(2\theta)  d\theta    -     \int_{\frac{\pi}{4}}^{\frac{\pi}{2}} \sqrt{1 - \frac{s^4 \sin^2(2\theta)}{s(\tau)^4 } } \cos(2\theta)  d\theta   
\end{equation} 
 We now note that  both terms are positive. Thus, the integral in the $\theta$ variable is signed (positive). Hence, we have

  \begin{align*} 
 | \Delta^2_h(r L_{12}^{c}(\Omega)(r))|_{L^{\infty}} &\geq  \Big|\alpha^{\frac{1}{q}}       h^{1+\alpha}\cos(\alpha^{\frac{1}{q}}\log(h))      \int_0^t  2 L^{s}_{12}(\Omega(\tau,h))  \int_{0}^{\frac{\pi}{2}}  \cos(2\theta(\tau))    \cos(2\theta) d\theta    d\tau   \Big| \\  
  &\geq c  \alpha^{\frac{1}{q}}  2       h^{1+\alpha} | \cos(\alpha^{\frac{1}{q}}\log(h))\ |       \int_0^t  2 L^{s}_{12}(\Omega(\tau,h))  \int_{0}^{\frac{\pi}{8}}  \cos(2\theta(\tau))    \cos(2\theta) d\theta    d\tau 
  \end{align*}
 Recall that the dynamic of the trajectories is hyperbolic. On $[0,\frac{\pi}{8}]$, the angular component of the inverse flow map is decreasing and the radial components  is increasing. Therefore, we have   
  $$
 \int_{0}^{\frac{\pi}{8}}  \cos(2\theta(\tau))    \cos(2\theta) d\theta = \int_{0}^{\frac{\pi}{8}}   \sqrt{1 - \frac{r^4  \sin^2(2\theta)}{r(\tau)^4 } } \cos(2\theta)  d\theta  \geq  c>0 $$ 
 Hence, we have 
   \begin{equation}\label{L12cFD}     | \Delta^2_h(r L_{12}^{c}(\Omega)(r))|_{L^{\infty}} \geq 
  c  \alpha^{\frac{1}{q}}         h^{1+\alpha} | \cos(\alpha^{\frac{1}{q}}\log(h))\ |       \int_0^t  L^{s}_{12}(\Omega(\tau,h))     d\tau.  
    \end{equation}
 Now we just need to compute the Besov norm $B^{1}_{\infty,q}$ to obtained our result.   Namely,
$$
|r L_{12}^{c}(\Omega)(t,r)|^q_{B^{1}_{\infty,q}}\geq  c  \int_{0}^{1}  | \Delta^2_h(r L_{12}^{c}(\Omega)(r))|^q_{L^{\infty}} \frac{1}{h^{2+q}} h   dh   \geq  c  \int_{0}^{1}  \alpha     h^{q\alpha} | \cos(\alpha^{\frac{1}{q}}\log(h))\ |^q      ( \int_0^t  L^{s}_{12}(\Omega(\tau,h))     d\tau  )^q   \frac{1}{h}    dh .
$$
 Therefore, from Proposition \ref{LowerBndL12}, and estimating the integral  in a similar manner to  Lemma \ref{Type2}, we have 
 $$
|r L_{12}^{c}(\Omega)(T)|^q_{B^{1}_{\infty,q}}   \geq     c     (\log(c \frac{t}{\alpha^{1-\frac{1}{q}}}+1) )^q .  
$$
Thus,  at $T(\alpha)=\alpha^{1-\frac{1}{q}} \log(c |\log(\alpha)|) \rightarrow 0$,  as $\alpha \rightarrow 0$, we have 

$$
|r L_{12}^{c}(\Omega)(T)|_{B^{1}_{\infty,q}}  \geq     c     (\log(c |\log(\alpha)|+1) ) \rightarrow \infty,  
$$
and this completes the proof.

   \end{proof}

\section{Useful Estimates on the Leading Order Model  }\label{UsefulLOME}
In this section, we prove  estimates for the leading order model, which we will use in Section \ref{RemainderrEstS} to control the error between the leading order model and the full 2d Euler equation. In Lemma \ref{LOMLcUpbnd}, we establish  upper bounds on $L^c_{12}$ operator. Then in Lemma \ref{LOMHodlerBnd} and Corollary \ref{corOmegaCirc}, we obtain H\"older type estimates on vorticity of the leading order model.

   \begin{lemma} \label{LOMLcUpbnd}

  For any $1<q<\infty$,  let $\Omega(t)$ be a solution to the leading order model:

\begin{equation*} 
 \dt \Omega+  2 L^s_{12}(\Omega) \sin(2\theta)  \dth \Omega- 2 rL^s_{12}(\Omega) \cos(2\theta) \dr \Omega =0,
\end{equation*} 

 with initial data: 
\begin{equation*}
\Omega_0(r,\theta) =r^{\alpha} \sin(\alpha^{\frac{1}{q}}\log(r))+\alpha^{\frac{1}{q}} r ^{\alpha} \sin(2\theta)
\end{equation*}

compactly supported on the ball of radius 1.  Then for $L^c_{12}(\Omega)$  defined as in \eqref{L12c},   there exist constants $c$ independent of $\alpha$ such that we have the following estimates 
 
\begin{equation} \label{L12cUpbnd}
| L_{12}^{c}(\Omega)(t)|_{L_{\infty}}  \leq    \frac{ct}{\alpha^{1-\frac{1}{q}}}, \,\, \text{and}  \,\, | L_{12}^{c}(\Omega)(t)|_{C^{\alpha}}  \leq    \frac{ct}{\alpha^{1-\frac{1}{q}}}.
\end{equation}

   \end{lemma}

 \begin{proof} We start in a similar manner as in Lemma   \ref{UpperBndL12},  Proposition \ref{LowerBndL12}, and Proposition  \ref{L12cBesovLowBnd}. Recall by setting the inverse flow map $\Phi^{-1}(t,r,\theta)=(r(t),\theta(t))$ as in  Lemma \ref{UpperBndL12}, and by studying the characteristic of the leading order model, we obtained  
 \begin{equation}
 (r(t))^2 |\sin(2\theta(t))|   =     r^2  |\sin(2\theta)| 
\end{equation} 
Similarly, as in  Lemma \ref{UpperBndL12},  shortening  the notation, we   write

 \begin{equation}\label{Xflow4L12cUpbnd}
   r(t)=re^{X(t)}, \quad \text{where} \quad X(t,r,\theta)=2\int_0^t(L_{12}(r(\tau))\cos(2\theta(\tau))) d\tau
\end{equation} 
We compute $L^{c}_{12} (\Omega)(r)$, also proceeding as in Proposition \ref{L12cBesovLowBnd}, where we use symmetry, compact support of the initial data, and fact that terms which are either radial or contain $\sin(2\theta(t))$ vanish, we have 
\begin{align}\label{L12cUpEx}
   L_{12}^{c}(\Omega)=c\int_{r}^{1}\int_{0}^{\frac{\pi}{2}}\frac{s^{\alpha}}{s} \Big(   \cos(\alpha^{\frac{1}{q}}\log(s) )  \alpha^{\frac{1}{q}} X(t)   + \alpha^{\frac{2}{q}}  X(t) E(s,X)  \Big)  \cos(2\theta) d\theta ds.
\end{align}
We obtained the above expression after expanding the terms $\cos(\alpha^{\frac{1}{q}} X(t))), \sin(\alpha^{\frac{1}{q}} X(t))), e^{\alpha X}$ by using Lemma \ref{UpperBndL12}, where  we have $|X(t)|_{L^{\infty}} \leq t \alpha^{-1+\frac{1}{q}}$ on the time scale $0\leq t \leq T(\alpha)=\alpha^{1-\frac{1}{q}}\log(c|\log(\alpha)|)$ (see Proposition \ref{L12cBesovLowBnd} for more details). In equation \eqref{L12cUpEx}, the term $E(s,X)$, which is bounded, contains the remaining terms from the expansion. Since this term has $\alpha^{\frac{2}{q}}$ in front, it is an error term that can be absorbed. Therefore, we have:

   $$
  | L_{12}^{c}(\Omega)| \leq \Big| c\int_{r}^{1}\int_{0}^{\frac{\pi}{2}}\frac{s^{\alpha}}{s}    \cos(\alpha^{\frac{1}{q}}\log(s) )  \alpha^{\frac{1}{q}} X(t)    \cos(2\theta) d\theta ds  \Big|. 
  $$
 From \eqref{Xflow4L12cUpbnd} we have

   $$
  | L_{12}^{c}(\Omega)| \leq \Big| c \int_0^t \int_{r}^{1}\int_{0}^{\frac{\pi}{2}}\frac{s^{\alpha}}{s}    \cos(\alpha^{\frac{1}{q}}\log(s) ) \alpha^{\frac{1}{q}}   2L^{s}_{12}(s(\tau))\cos(2\theta(\tau))      \cos(2\theta)  d\theta ds d\tau  \Big|. 
  $$
 To simplify notation, we write:
   $$
  \Theta(t,s):=\int_{0}^{\frac{\pi}{2}} \cos(2\theta(\tau))      \cos(2\theta)  d\theta. 
  $$
We know from Proposition \ref{L12cBesovLowBnd} and equation  \eqref{ThetaInt} that  this integral   is positive. Thus, we can write

\begin{align}\label{L12cdsdtau}
  | L_{12}^{c}(\Omega)| \leq \Big| c \int_0^t  \int_{r}^{1}  \frac{s^{\alpha}}{s}    \cos(\alpha^{\frac{1}{q}}\log(s) ) \alpha^{\frac{1}{q}} 2L^{s}_{12}(s(\tau))   \Theta(t,s)  ds d\tau   \Big|.
\end{align}
 Now let us focus on the  radial integral. We change variables $u= \frac{ \log(s)}{\log(r)} \implies s=e^{ \log(r) u}   $. We do this in order to recast \eqref{L12cdsdtau}  into an oscillatory  integral. Hence, we have  
\begin{align*} 
  \int_{r}^{1} \frac{s^{\alpha}}{s}    \cos(\alpha^{\frac{1}{q}}\log(s) )& \alpha^{\frac{1}{q}} 2L^{s}_{12}(s(\tau))   \Theta(\tau,s)   ds\\&=-c  \log(r)  \int_{0}^{1}e^{\alpha \log(r)u}    \cos(\alpha^{\frac{1}{q}}\log(r) u)  \alpha^{\frac{1}{q}}     L^s_{12}(e^{ \log(r) u(\tau)} )   \Theta(\tau ,e^{ \log(r) u}  )  du . 
\end{align*}
 Setting   $\lambda=\alpha^{\frac{1}q}|\log(r)|$, we can further rewrite this as:  
  \begin{align}\label{L12cOscB}
  \int_{r}^{1} \frac{s^{\alpha}}{s}    \cos(\alpha^{\frac{1}{q}}\log(s) ) \alpha^{\frac{1}{q}} 2L^{s}_{12}(s(\tau))   \Theta(\tau,s)   ds= c \frac{\lambda}{\alpha^{\frac{1}{q}}}  \int_{0}^{1}e^{-\alpha^{1-\frac{1}{q}}  \lambda u}    \cos( \lambda u)  \alpha^{\frac{1}{q}}     L^s_{12}(e^{ \log(r) u(\tau)} )   \Theta(\tau ,e^{ \log(r) u}  )  du . 
\end{align}
  Now we have $  |\Theta(t,s)|_{L^{\infty}} \leq c$ is bounded uniformly, and from Lemma \ref{UpperBndL12}, $  \alpha^{1-\frac{1}{q}} L_{12}(\Omega)$ is also bounded uniformly,   $\alpha^{1-\frac{1}{q}} |   L_{12} |_{L^{\infty}}<c$. Thus,  as  $r \rightarrow 0$,  we have $\lambda=\alpha^{\frac{1}{q}}|\log(r)| \rightarrow \infty$. Therefore,  
  
    $$
  \bigg|  \int_{0}^{1} e^{-\alpha^{1-\frac{1}{q}}  \lambda u}    \cos( \lambda u)  \alpha^{1-\frac{1}{q}}     L^s_{12}(e^{ \log(r) u(\tau)} )   \Theta(\tau ,e^{ \log(r) u}  )  du       \bigg|    \leq \frac{c}{\lambda}.
  $$
Hence, we have 
\begin{align}\label{L12cOscF}
  \bigg|  c \frac{\lambda}{\alpha^{1-\frac{1}{q}}}    \int_{0}^{1} e^{-\alpha^{1-\frac{1}{q}}  \lambda u}    \cos( \lambda u)  \alpha^{1-\frac{1}{q}}     L^s_{12}(e^{ \log(r) u(\tau)} )   \Theta(\tau ,e^{ \log(r) u}  )  du       \bigg|    \leq \frac{c}{\alpha^{1-\frac{1}{q}}}. 
\end{align}
  Thus from \eqref{L12cdsdtau},  \eqref{L12cOscB}, and  \eqref{L12cOscF}, we have our desired result:

     $$
  | L_{12}^{c}(\Omega)|_{L^{\infty}}   \leq \frac{ct}{\alpha^{1-\frac{1}{q}}},
  $$
  the $C^{\alpha}$ estimate follows similarly, and hence this completes the proof.

  \end{proof}

  \begin{lemma} \label{LOMHodlerBnd}

    For any $1<q<\infty$,   let $\Omega(t)$ be a solution to the leading order model:

\begin{equation*} 
 \dt \Omega+  2 L^s_{12}(\Omega) \sin(2\theta)  \dth \Omega- 2 rL^s_{12}(\Omega) \cos(2\theta) \dr \Omega =0,
\end{equation*} 

 with initial data: 
\begin{equation*} 
\Omega_0(r,\theta) =r^{\alpha} \sin(\alpha^{\frac{1}{q}}\log(r))+\alpha^{\frac{1}{q}} r ^{\alpha} \sin(2\theta)
\end{equation*}

compactly supported on the ball of radius 1. Then,   there exist constants $c$ independent of $\alpha$ such that we have the following estimates 
 
\begin{equation}  
| \dth \Omega|_{C^{\alpha}}  \leq  c   \alpha^{\frac{1}{q}} e^{\frac{ct}{\alpha^{1-\frac{1}{q}}}}, \,\,\, \text{and}  \,\,\, | r \dr \Omega|_{C^{\alpha}}  \leq     c   \alpha^{\frac{1}{q}} e^{\frac{ct}{\alpha^{1-\frac{1}{q}}}}.
\end{equation}

   \end{lemma}
   
   \begin{proof}
   We observe that 
   $$
   |\dth  \Omega_{0}|_{C^{\alpha}} \leq c \alpha^{\frac{1}{q}} , \quad \text{and}   \quad   |r \dr \Omega_{0}|_{C^{\alpha}} \leq c \alpha^{\frac{1}{q}},
   $$  
   and since the vorticity is being transported, we have 
   $$
   \Omega(r,t,\theta)=\Omega_{0}(\Phi^{-1}_t(r,\theta)).
   $$
  Thus, from transport estimates and Lemma \ref{UpperBndL12}, the result follows: 
\begin{equation}  
| \dth \Omega|_{C^{\alpha}}  \leq  c   \alpha^{\frac{1}{q}} e^{\frac{ct}{\alpha^{1-\frac{1}q}}}, \,\,\, \text{and}  \,\,\, | r \dr \Omega|_{C^{\alpha}}  \leq     c   \alpha^{\frac{1}{q}} e^{\frac{ct}{\alpha^{1-\frac{1}{q}}}}.
\end{equation} 

   \end{proof}
 The following corollary follows from the definition of $\mathring{C}^{1,\alpha}$ norm. Recall that: 
   
   $$
   |\Omega|_{\mathring{C}^{1,\alpha}}=    |\Omega|_{C^{\alpha}}+   |\dth  \Omega|_{C^{\alpha}}+|r \dr \Omega|_{C^{\alpha}}. 
   $$    
   \begin{corollary}\label{corOmegaCirc}
For any $1<q<\infty$,   let $\Omega(t)$ be a solution to the leading order model:

\begin{equation*} 
 \dt \Omega+  2 L^s_{12}(\Omega) \sin(2\theta)  \dth \Omega- 2 rL^s_{12}(\Omega) \cos(2\theta) \dr \Omega =0,
\end{equation*} 

 with initial data 
\begin{equation*} 
\Omega_0(r,\theta) =r^{\alpha} \sin(\alpha^{\frac{1}{q}}\log(r))+\alpha^{\frac{1}{q}} r ^{\alpha} \sin(2\theta),
\end{equation*} 

compactly supported on the ball of radius 1. Then   there exist constants $c$ independent of $\alpha$ such that we have the following estimate:

 \begin{equation} \label{OmegaCirc}
 |\Omega|_{\mathring{C}^{1,\alpha}}   \leq      c    e^{\frac{ct}{\alpha^{1-\frac{1}{q}}}}. 
\end{equation}

\end{corollary}

\section{Remainder Estimate } \label{RemainderrEstS}
In this section, we estimate the error between the leading order model and the  Euler equation. Recall   that the Euler equation (in polar coordinates) satisfies:
$$\dt \omega+\frac{1}{r} \dr \psi \dth \omega-\frac{1}{r}\dth \psi \dr\omega=0$$ 

We will write the solution to the Euler equation as $\omega=\omega_r+ \Omega$. where $\Omega$ solves the leading order model, and $\omega_r$ solves the rest.  Now for the stream function, we have  $\psi(\omega)=\psi(\omega_r)+ \psi(\Omega)$, and   recall from Proposition \ref{ElgindiCcirc} and Remark \ref{PsiReminderEst}, we can write: 

\begin{equation} \psi_r(\Omega)=\psi(\Omega)-r^2L^{s}_{12}(\Omega)(r) \sin(2\theta)-r^2L^{c}_{12}(\Omega)(r) \cos(2\theta),
 \end{equation} 
and since our leading order model is: 

\begin{equation} 
 \dt \Omega+  2 L^s_{12}(\Omega) \sin(2\theta)  \dth \Omega- 2 rL^s_{12}(\Omega) \cos(2\theta) \dr \Omega =0.
\end{equation} 
Thus, we can write the equation for the reminder $\omega_r=\omega- \Omega $ with $\omega_r|_{t=0}=0$ as follows:

\begin{equation}\label{RemainderEq}
\begin{split}
  \dt \omega_r& + \frac{1}{r}  \Big(2rL^s_{12}(\Omega)(r)\sin(2\theta)+ 2rL^c_{12}(\Omega)(r)\cos(2\theta)   + \dr \psi_r(\Omega)+\dr \psi(\omega_r)\Big) \dth \omega_r
   \\ &-\frac{1}{r}   \Big( 2r^2L^s_{12}(\Omega)(r)\cos(2\theta)- 2r^2L^c_{12}(\Omega)(r)\sin(2\theta)+\dth \psi_r(\Omega) + \dth \psi(\omega_r) \Big) \dr \omega_r  \\
  &+\frac{1}{r}  \Big(   2rL^c_{12}(\Omega)(r)\cos(2\theta)+   \dr \psi_r(\Omega)+\dr \psi(\omega_r)  \Big)   \dth \Omega \\ &-\frac{1}{r}   \Big( - 2r^2L^c_{12}(\Omega)(r)\sin(2\theta)+ \dth \psi_r(\Omega) + \dth \psi(\omega_r)  \Big)   \dr \Omega =0.
\end{split}
\end{equation}

\begin{proposition}\label{RemainderEqP}
Let $\omega_r$ be a solution be a solution to \eqref{RemainderEq} with $\omega_r|_{t=0}=0$, then we have the following:
\begin{equation}\label{RemainderEqest}
|\omega_r|_{C^{\alpha}} \leq c \alpha^{\frac{1}{q}}, 
\end{equation}
on the time interval $0 \leq t\leq  T(\alpha)=c \alpha^{1-\frac{1}{q}}\log(c |\log(\alpha)|) $.

\end{proposition}

\begin{proof}Since $\omega_r$ satisfies \eqref{RemainderEq},   composing with the flow map generates by the velocity field and estimating the  H\"older norm, we obtain: 
\begin{align*}
  \frac{d}{dt}|\omega_r|_{C^{\alpha}} & \leq c  \Big(  |L^c_{12}(\Omega) |_{C^{\alpha}}  +   | \frac{1}{r} \dr \psi_r(\Omega)|_{C^{\alpha}}  +  | \frac{1}{r} \dr \psi(\omega_r)|_{C^{\alpha}} \Big)   |\dth \Omega|_{C^{\alpha}} \\
&+ c \Big( |L^c_{12}(\Omega)|_{C^{\alpha}} +   | \frac{1}{r^2}   \dth \psi_r(\Omega)  |_{C^{\alpha}} +  |\frac{1}{r^2}    \dth \psi(\omega_r) |_{C^{\alpha}} \Big)  |r \dr \Omega|_{C^{\alpha}} .
\end{align*}
Now we start estimating each term. Since the vorticity is compactly supported on the ball of radius 1,  the support of the vorticity on our time scale will be controlled. For instance, on the ball of radius 2. Thus, when we do the Schauder estimates, to simplify the notation,  we will drop $L^{1}$ term.

\textbf{Estimate on  $ |L^c_{12}(\Omega)|_{C^{\alpha}}$  }

This estimates follows from Lemma \ref{LOMLcUpbnd}. Hence, we have    
\begin{align}\label{L12cRemEst}
& |L^c_{12}(\Omega)|_{C^{\alpha}}  \leq    c e^{\frac{t}{\alpha^{1-\frac{1}{q}}}}.   
\end{align}

\textbf{Estimate on $|\frac{1}{r} \dr \psi_r(\Omega)|_{C^{\alpha}}$ and $| \frac{1}{r^2}  \dth \psi_r(\Omega) |_{C^{\alpha}}$} 

This estimates follows from Proposition \ref{ElgindiCcirc} and Corollary \ref{corOmegaCirc}.  Thus, we have    
\begin{align}
&|\frac{1}{r}\dr \psi_r(\Omega)|_{C^{\alpha}}\leq c    |\Omega|_{\mathring{C}^{1,\alpha}}  \leq   c e^{\frac{ct}{\alpha^{1-\frac{1}{q}}}}, \label{drPsirRemEst}    \\
& |\frac{1}{r^2} \dth \psi_r(\Omega) |_{C^{\alpha}} \leq    c |\Omega|_{\mathring{C}^{1,\alpha}}  \leq        ce^{\frac{ct}{\alpha^{1-\frac{1}{q}}}}.  \label{dthPsirRemEst}
\end{align}

\textbf{Estimate on $|\frac{1}{r} \dr \psi(\omega_r)|_{C^{\alpha}}$ and $| \frac{1}{r^2}  \dth \psi(\omega_r) |_{C^{\alpha}}$ } 

This estimates follows from standard elliptic Schauder estimates:  
    \begin{align}
&|\frac{1}{r}\dr \psi(\omega_r)|_{C^{\alpha}}\leq \frac{c}{\alpha}|\omega_r|_{C^{\alpha}},   \label{dthPsiRemEst} \\
& |\frac{1}{r^2} \dth \psi(\omega_r) |_{C^{\alpha}} \leq \frac{c}{\alpha}|\omega_r|_{C^{\alpha}}  . \label{dthPsRemiEst}
\end{align}

\textbf{Estimate on $|\dth \Omega|_{C^{\alpha}}$ and $|r \dr \Omega|_{C^{\alpha}}$} 

This estimates follows from Lemma \ref{LOMHodlerBnd}. Therefore, we have 
\begin{align}
&|\dth \Omega|_{C^{\alpha}} \leq  c \alpha^\frac{1}{q}   e^{\frac{ct}{\alpha^{1-\frac{1}{q}}}} ,\label{dthOmegaRemEst} \\
&|r \dr \Omega|_{C^{\alpha}} \leq   c \alpha^\frac{1}{q}   e^{\frac{ct}{\alpha^{1-\frac{1}{q}}}} .\label{drOmegaRemEst} 
 \end{align}

\textbf{Total Estimate}

Now from \eqref{L12cRemEst}, \eqref{drPsirRemEst}, \eqref{dthPsirRemEst}, \eqref{dthPsiRemEst}, \eqref{dthPsRemiEst}, \eqref{dthOmegaRemEst}, and \eqref{drOmegaRemEst}, we have 
 \begin{align*}
  \frac{d}{dt}|\omega_r|_{C^{\alpha}} & \leq  \Big(      c e^{\frac{ct}{\alpha^{1-\frac{1}{q}}}}   +  \frac{c}{\alpha}|\omega_r|_{C^{\alpha}}  \Big)   c \alpha^\frac{1}{q}   e^{\frac{ct}{\alpha^{1-\frac{1}{q}}}}.  
  \end{align*}

Now since $\omega_r|_{t=0}=0$, by Gronwall inequality, we have

  \begin{align*}
   |\omega_r|_{C^{\alpha}} & \leq \Big( \int_0^t  c \alpha^\frac{1}{q}   e^{ \frac{\tau}{\alpha^{1-\frac{1}{q}}}}     \Big) \exp{\Big( \int_0^t   \frac{c}{\alpha^{1-\frac{1}{q}}}  e^{\frac{\tau}{\alpha^{1-\frac{1}{q}}}}   }\Big)  \leq   \  c \alpha   e^{\frac{t}{\alpha^{1-\frac{1}{q}}}}     \exp{\Big(    c   e^{\frac{t}{\alpha^{1-\frac{1}{q}}}}   }\Big) . 
  \end{align*}

 Choosing our time $T(\alpha)=c \alpha^{1-\frac{1}{q}}\log(c |\log(\alpha)|) $ for $c$ small enough and independent of $\alpha$, we obtain   the following:   
\begin{align*}
   |\omega_r|_{C^{\alpha}} &  \leq c   \alpha^{ \frac{1}{q} +\eta} \leq  c  \alpha^{ \frac{1}{q}},
  \end{align*}
for $0\leq t \leq  T(\alpha)$, and for some $\eta>0$, and this completes our remainder estimates. 
\end{proof}

  \section{Main Result  }\label{Main}
In this section, we prove our  result:

  \begin{theorem*} (Strong ill-posedness in ${B^{1}_{\infty,q}}$): For any $\delta>0$,  $1<q<\infty$, and $0<\alpha<1$, there exist initial velocity data  $u_0^{\alpha,q,\delta} \in C^{1,\alpha}(\mathbb{R}^2)$ such that the solution to the 2d Euler equation  satisfies the following: 
  
  $$|u_0|_{B^{1}_{\infty,q}}=\delta,  \hspace{0.3 cm } \text{but}  \hspace{0.3 cm }  \sup_{0\leq t \leq T(\alpha)}|u(t)|_{B^{1}_{\infty,q}} \geq   c\log(c |\log(\alpha)|), $$ 
  
  where 
  $T(\alpha)=c{\alpha^{1-\frac{1}{q}} \log(c |\log(\alpha)|) }$ and c is independent of $\alpha$.  
  \end{theorem*}
  \begin{remark}
Taking $\alpha \rightarrow 0$ implies that   $T(\alpha)=c{\alpha^{1-\frac{1}{q}} \log(c |\log(\alpha)|) }  \rightarrow 0$, and $  |u|_{B^{1}_{\infty,q}} \rightarrow \infty$.  Thus, we have  strong ill-posedness.  
   \end{remark}
  \begin{proof}
  
  We consider the following initial data
   \begin{equation}\label{MTID}
  \omega_0(r,\theta)=\Omega_0(r,\theta) =r^{\alpha} \sin(\alpha^{\frac{1}{q}}\log(r))+\alpha^{\frac{1}{q}} r ^{\alpha} \sin(2\theta)
  \end{equation}
  compactly supported in the ball of radius 1. 
  Recall that we can write the solution to 2d the Euler equation as:
  $$
  \omega(t)=\Omega(t)+\omega_r(t)
  $$
  where $\Omega$ is a solution to the leading order model \eqref{LOMeqExp} with initial data $\Omega|_{t=0}= \omega_0(r,\theta)$ in \eqref{MTID}, and $\omega_r$ is the remainder  satisfying \eqref{RemainderEq}   with   $\omega_r|_{t=0}=0$ From Proposition \ref{RemainderEqP}, we have 
  
   \begin{equation}\label{RemainderEqEstMT}
  |\omega_r|_{C^{\alpha}} \leq c \alpha^{\frac{1}{q} }
  \end{equation}
 Now we can write $\psi(\omega)=\psi(\Omega)+ \psi(\omega_r)$ and recall that the velocity is defined as follows: $$u=\nabla^{\perp} \cdot \psi  $$ Without the loss of generality, we consider $u_1$ (same can be done for $u_2$), and  write 
  
  $$
  u_{1}(\omega)= u_{1}(\Omega)+   u_{1}(\omega_r)
  $$ From \eqref{RemainderEqEstMT}, we have    $|\omega_r|_{C^{\alpha}} \leq c \alpha^{\frac{1}{q} }$, on the time scale  $T(\alpha)=c{\alpha^{1-\frac{1}{q}} \log(c |\log(\alpha)|) }$.  Hence, we have $|\omega_r|_{B^{\alpha}_{\infty,\infty}} \leq c \alpha^{\frac{1}{q} }$, which implies  $|\omega_r|_{B^{0}_{\infty,q}} \leq c $, with $c$ is independent of $\alpha$, on the same time interval.  Since the Riesz operators are bounded on   Besov spaces, we have  $u_1(\omega_r)$ is bounded in    ${B^{1}_{\infty,q}}.$ Namely,

   \begin{equation}\label{VeloctyReminder}
  |u_1(\omega_r)|_{B^{1}_{\infty,q}} \leq c,  
    \end{equation}
    where $c$ is independent of $\alpha$. Now we will estimate the velocity generated from the leading order model.
     Recall that $\dy=\sin(\theta) \dr + \frac{\cos(\theta)}{r} \dth$, hence the velocity generated from the leading order model will be
   \begin{align*} 
   u_1(\Omega) &=-\sin(\theta)\Big(2rL^{s}_{12}(\Omega)\sin(2\theta)+ 2rL^{c}_{12}(\Omega)\cos(2\theta)+r^2\dr L^{s}_{12}(\Omega)\sin(2\theta)+ r^2\dr L^{c}_{12}(\Omega)\cos(2\theta)\Big)\\
   &-\frac{\cos(\theta)}{r}\Big(2r^2L^{s}_{12}(\Omega)\cos(2\theta)- 2r^2L^{c}_{12}(\Omega)\sin(2\theta)) \Big).
     \end{align*}
   Reorganizing the terms, we have 
         
            \begin{align*} 
   u_1(\Omega) &= 2rL^{c}_{12}(\Omega)  \sin(\theta)  - 2rL^{s}_{12}(\Omega) \cos(\theta)    \\
   &  +  \frac{1}{2} r^2\dr L^{s}_{12}(\Omega) \Big(\cos(\theta))-\cos(3\theta)  \Big)  -     \frac{1}{2} r^2\dr L^{c}_{12}(\Omega)  \Big(  \sin(3\theta)-\sin(\theta) \Big).
     \end{align*} 
    Recall that the second finite difference is $\Delta^2_{h}u_1=u_1(x+2h)-2u_1(x+h)+u_1(x)$, where here  $h=(h_1,h_2)$ is a vector. Thus, we have 
     $$
     |\Delta^2_{h}u_1(\Omega)|_{L^{\infty}} \geq |  u_1(\Omega)(2h)-2u_1(\Omega)(h)|.
     $$
     
To simplify the notation, we will write $h=\sqrt{h_1^2+h_2^2}$ to denote the radial component. In addition, we will  define the following terms:
     \begin{align*}
   u_1(\Omega)(2h)-2u_1(\Omega)(h) & = \underbrace{ 4h  \Big( L^{c}_{12}(\Omega)(2h)- L^{c}_{12}(\Omega)(h)  \Big) \sin(\theta)}_{K_1} + \underbrace{-4h  \Big( L^{s}_{12}(\Omega)(2h)- L^{s}_{12}(\Omega)(\Omega)(h)  \Big) \cos(\theta) }_{K_2}    \\
   & \underbrace{ + \Big(2 h^2\dr L^{s}_{12}(\Omega)(2h) - h^2\dr L^{s}_{12}(\Omega)(h)  \Big)   \Big(\cos(\theta))-\cos(3\theta)  \Big) }_{K_3}\\ 
    &  \underbrace{- \Big(2 h^2\dr L^{c}_{12}(\Omega)(2h) - h^2\dr L^{c}_{12}(\Omega)(h)  \Big)  \Big(  \sin(3\theta)-\sin(\theta) \Big)}_{K_4}. 
     \end{align*} 
    
  When computing the Besov $B^{1}_{\infty,q}$ norm of the above terms, we will encounter integrals of the following types:
 $$
 \int_{B_1(0)}  |f(h)g(\theta)|^q \frac{1}{|h|^{2+q}}  h  dh d\theta, $$
   
 where $g(\theta)$ is going to be either $\sin(\theta), \cos(\theta), \sin(3\theta), \text{or} \cos(3\theta)$. Thus, they can be estimated as follows:
 
  $$
 \int_{B_1(0)} |f(h)g(\theta)|^q \frac{1}{h^{2+q}}  h   dh d\theta \geq   \int_{0}^1\int_{c_1\pi}^{c_1\pi}|f(h)g(\theta)|^q \frac{1}{h^{2+q}} h   d\theta dh  \geq   c   \int_{0}^1 |f(h)|^q \frac{1}{h^{2+q}} h  dh.   $$
 
Similarly, we have

  $$
 \int_{B_1(0)} |f(h)g(\theta)|^q \frac{1}{h^{2+q}}  h  dh d\theta  \leq   c   \int_{0}^1 |f(h)|^q \frac{1}{h^{2+q}} h  dh.   $$

Hence, it suffice to just do the radial integral. Now, we start estimating each term

     \textbf{Estimate on $K_1$}
     
     From Proposition \ref{L12cBesovLowBnd}, we know that   
     
    $$
|r L_{12}^{c}(\Omega)(t)|_{B^{1}_{\infty,q}}
   \geq     c     \log(c\frac{t}{\alpha^{1-\frac{1}{q}}}+1).  $$
   
   Namely, from equation \eqref{L12cFD}, we have

      $$    | \Delta^2_h(r L_{12}^{c}(\Omega)(r))|_{L^{\infty}} \geq 
  c  \alpha^{\frac{1}{q}}         h^{1+\alpha} | \cos(\alpha^{\frac{1}{q}}\log(h))\ |       \int_0^t  L^{s}_{12}(\Omega(\tau,h))     d\tau |.  $$

 Thus, in the same manner, it follows that

   \begin{equation}\label{K1est}
 c \int_0^1 |K_1|^{q}  \frac{1}{h^{2+q}} h   dh   \geq  (c     \log(c\frac{t}{\alpha^{1-\frac{1}{q}}}+1))^q, 
   \end{equation}
   
   on the time scale $T(\alpha)=c{\alpha^{1-\frac{1}{q}} \log(c |\log(\alpha)|) }$.
     
          \textbf{Estimate on $K_2$ and $K_3$ }
          
 From Proposition \ref{LowerBndL12}, using equation \eqref{L12smp0},  we have

                \begin{equation*}
  L_{12}^{s}(\Omega)(t)=c\int_{r}^{1}\int_{0}^{\frac{\pi}{2}}\frac{ 1}{s} \Big( \alpha^{\frac{1}{q}}  s^{\alpha}e^{\alpha X(t)} \sin(2\theta(t)) \Big)    \sin(2\theta) d\theta ds.
 \end{equation*} 
 
 Thus, by taking $\alpha$ small enough, we have 
 
 $$
 |K_2| \leq  c h \int_{h}^{2h} \frac{ 1}{s}   \alpha^{\frac{1}{q}}  s^{\alpha} ds \leq c \alpha^{\frac{1}{q}} h^{1+\alpha}
 $$
 
 on the same time scale    $T(\alpha)$. Hence, we obtain 
   \begin{equation}\label{K2est}
 c \int_0^1 |K_2|^{q}  \frac{1}{h^{1+q}}     dh   \leq c,
   \end{equation} 
 
where $c$ is independent  of $\alpha$.   Similarly, using equation \eqref{L12smp0}, we have

                \begin{equation*} 
  r \dr L_{12}^{s}(\Omega)(r)=-c \int_{0}^{\frac{\pi}{2}}  \Big( \alpha^{\frac{1}{q}}  r^{\alpha}e^{\alpha X(t)} \sin(2\theta(t)) \Big)    \sin(2\theta) d\theta.
 \end{equation*} 
 
Thus, we obtain

   \begin{equation*}  
 |K_3|  \leq c \alpha^{\frac{1}{q}} |h|^{1+\alpha},
   \end{equation*} 
 
 which similarly gives

   \begin{equation}\label{K3est}
  \int_0^1 |K_3|^q  \frac{1}{h^{1+q}}     dh   \leq c, 
   \end{equation} 
   
with $c$ independent  of $\alpha$.

                    \textbf{Estimate on $K_4$}
    
     From Proposition \ref{L12cBesovLowBnd}, using equation \eqref{L12cExp2},  we have     
\begin{align*} 
   L_{12}^{c}(\Omega)=c\int_{r}^{1}\int_{0}^{\frac{\pi}{2}}\frac{s^{\alpha}}{s} \Big(   \cos(\alpha^{\frac{1}{q}}\log(s) )  \alpha^{\frac{1}{q}} X(t)   + \alpha^{\frac{2}{q}}  X(t) E(s,X)  \Big)  \cos(2\theta) d\theta ds
\end{align*}

    Thus, by taking   $\alpha$ small enough, we have 
    $$|K_4|\leq \alpha^{\frac{1}{q}} |X(t)|    |h|^{1+\alpha}  \Big|   \cos(\alpha^{\frac{1}{q}}\log(2|h|) )      -   \cos(\alpha^{\frac{1}{q}}\log(|h|) )    \Big| \leq    \alpha^{\frac{2}{q}}  |X(t)|    |h|^{1+\alpha}  $$
    
    on the time scale $T(\alpha)=c{\alpha^{1-\frac{1}{q}} \log(c |\log(\alpha)|) }$. Hence, we obtain     
       \begin{equation}\label{K4est}
    \int_0^1 |K_4|^{q}  \frac{1}{h^{1+q}}   dh   \leq c 
   \end{equation} 
   
   where $c$ is independent  of $\alpha$.   
     
     \textbf{Final Estimate}

From \eqref{K1est}, \eqref{K2est}, \eqref{K3est}, and \eqref{K4est}, we observe that the main term is \eqref{K1est}. Thus, the remaining terms can be absorbed when computing the Besov norm. Therefore, we have
     
     $$
| u_1(\Omega) |_{B^{1}_{\infty,q}} \geq  c   \log(c \frac{t}{\alpha^{1-\frac{1}{q}}}+1)   
$$
on our time interval $0\leq t \leq T(\alpha)=c{\alpha^{1-\frac{1}{q}} \log(c |\log(\alpha)|) }$. Thus,  from \eqref{VeloctyReminder}, we have that the velocity of the  Euler equation satisfies 

 $$\sup_{0\leq t \leq T(\alpha)}|u_1(\omega)(t)|_{B^{1}_{\infty,q}} \geq  c \log(c |\log(\alpha)|)$$
   where $T(\alpha)=c{\alpha^{1-\frac{1}{q}} \log(c |\log(\alpha)|) }  \rightarrow 0$, and  $|u_1(\omega) |_{B^{1}_{\infty,q}}  \rightarrow \infty$ as $\alpha \rightarrow 0$. This gives strong ill-posedness. Now, since the constants $c$ are independent of $\alpha$, rescaling that data by $\delta>0$ completes the proof.  

\end{proof}
The following corollary follows from the fact that we can approximate the initial $C^{\alpha}$ vorticity by smooth  data, and that $C^{\alpha}\subset B^{0}_{\infty,q}$.

   \begin{corollary*}  For any $\delta>0$,  $1<q<\infty$, and $0<\alpha<1$,  there exist smooth initial velocity   $u_0^{\delta,\alpha,q} \in C^{\infty}(\mathbb{R}^2)$ such that the solution to the 2d Euler equation  satisfies the following: 
  
  $$|u_0|_{B^{1}_{\infty,q}}=\delta,  \hspace{0.3 cm } \text{but}  \hspace{0.3 cm }  \sup_{0\leq t \leq T(\alpha)}|u(t)|_{B^{1}_{\infty,q}} \geq   c \log(c |\log(\alpha)|), $$ 
  
  where 
  $T(\alpha)=c{\alpha^{1-\frac{1}{q}} \log(c |\log(\alpha)|) }$ and c is independent of $\alpha$.  
  \end{corollary*}

\section*{Acknowledgements}
I would like to thank Professor Tarek M. Elgindi. I am grateful for all his important suggestions and feedback. I would also like to thank Professor Theodore D. Drivas for his valuable feedback.

\vskip 0.3 cm

Mathematics Department, Stony Brook University, Stony Brook NY, 11794
 
\textit{Email address:} karim.shikhkhalil@stonybrook.edu

\end{document}